\documentclass[11pt,a4paper]{amsart}
\usepackage{amssymb,amsthm}
\usepackage{bm}

\usepackage[truedimen,margin=30truemm]{geometry}

\usepackage[pdftex,colorlinks=true,bookmarks=true,citecolor=blue,linkcolor=blue,pdfauthor={},pdftitle={}]{hyperref}

\theoremstyle{plain}
\newtheorem{theorem}{Theorem}[section]
\newtheorem{lemma}[theorem]{Lemma}

\theoremstyle{definition}

\theoremstyle{remark}
\newtheorem{remark}[theorem]{Remark}

\numberwithin{equation}{section}

\begin{document}

\title[Nonlinear stability of the two-jet Kolmogorov type flow]{Nonlinear stability of the two-jet Kolmogorov type flow on the unit sphere under a perturbation with nondissipative part}

\author[T.-H. Miura]{Tatsu-Hiko Miura}
\address{Graduate School of Science and Technology, Hirosaki University, 3, Bunkyo-cho, Hirosaki-shi, Aomori, 036-8561, Japan}
\email{thmiura623@hirosaki-u.ac.jp}



\begin{abstract}
  We consider the vorticity form of the Navier--Stokes equations on the two-dimensional unit sphere and study the nonlinear stability of the two-jet Kolmogorov type flow which is a stationary solution given by the zonal spherical harmonic function of degree two.
  In particular, we assume that a perturbation contains a nondissipative part given by a linear combination of the spherical harmonics of degree one and investigate the effect of the nondissipative part on the long-time behavior of the perturbation through the convection term.
  We show that the nondissipative part of a weak solution to the nonlinear stability problem is preserved in time for all initial data.
  Moreover, we prove that the dissipative part of the weak solution converges exponentially in time towards an equilibrium which is expressed explicitly in terms of the nondissipative part of the initial data and does not vanish in general.
  In particular, it turns out that the asymptotic behavior of the weak solution is finally determined by a system of linear ordinary differential equations.
  To prove these results, we make use of properties of Killing vector fields on a manifold.
  We also consider the case of a rotating sphere.
\end{abstract}

\maketitle

\section{Introduction} \label{S:Intro}
Let $S^2$ be the two-dimensional (2D) unit sphere in $\mathbb{R}^3$.
We consider the Navier--Stokes equations
\begin{align} \label{E:NS_Intro}
  \partial_t\mathbf{u}+\nabla_{\mathbf{u}}\mathbf{u}-\nu(\Delta_H\mathbf{u}+2\mathbf{u})+\nabla p = \mathbf{f}, \quad \mathrm{div}\,\mathbf{u} = 0 \quad\text{on}\quad S^2\times(0,\infty).
\end{align}
Here $\mathbf{u}$ is the tangential velocity field of a fluid, $p$ is the pressure, and $\mathbf{f}$ is a given external force.
Also, $\nu>0$ is the viscosity coefficient, $\nabla_{\mathbf{u}}\mathbf{u}$ is the covariant derivative of $\mathbf{u}$ along itself, $\Delta_H$ is the Hodge Laplacian via identification of vector fields and one-forms, and $\nabla$ and $\mathrm{div}$ are the gradient and the divergence on $S^2$.
Note that the viscous term in \eqref{E:NS_Intro} contains the zeroth order term since it is twice of the deformation tensor $\mathrm{Def}\,\mathbf{u}$:
\begin{align} \label{E:div_Def}
  2\,\mathrm{div}\,\mathrm{Def}\,\mathbf{u} = \Delta_H\mathbf{u}+\nabla(\mathrm{div}\,\mathbf{u})+2\,\mathrm{Ric}(\mathbf{u}) = \Delta_H\mathbf{u}+\nabla(\mathrm{div}\,\mathbf{u})+2\mathbf{u}.
\end{align}
Here $\mathrm{Ric}\equiv1$ is the Ricci curvature of $S^2$.
There are many works on the Navier--Stokes equations on spheres and manifolds with this kind of viscous term (see e.g. \cite{Tay92,Prie94,Nag99,MitTay01,DinMit04,KheMis12,ChaCzu13,ChaYon13,SaTaYa13,SaTaYa15,Pie17,KohWen18,PrSiWi20,SamTuo20}).
Also, several authors studied the Navier--Stokes equations on manifolds with viscous term replaced by $\nu\Delta_H\mathbf{u}$ in analogy of the flat domain case (see e.g. \cite{IliFil88,Ili90,TemWan93,CaRaTi99,Ili04,Wir15,Lic16,Ski17}).
We refer to \cite{EbiMar70,Aris89,DuMiMi06,Tay11_3,ChCzDi17} for the identity \eqref{E:div_Def} and the choice of the viscous term in the Navier--Stokes equations on manifolds.
A crucial difference due to the choice of the viscous term is the presence of nondissipative vector fields.
The viscosity always works if one chooses $\nu\Delta_H\mathbf{u}$ since $S^2$ does not admit nontrivial harmonic forms.
On the other hand, if one takes $\nu(\Delta_H\mathbf{u}+2\mathbf{u})$, then the viscosity does not work for tangential vector fields of the form $X(x)=\mathbf{a}\times x$, $x\in S^2$ with any $\mathbf{a}\in\mathbb{R}^3$.
Note that this $X$ is a Killing vector field on $S^2$, and in general Killing vector fields are nondissipative stationary solutions to the Navier--Stokes equations on a manifold with viscous term given by \eqref{E:div_Def}.
In fact, the problem of the existence and uniqueness of solutions is not so affected by the choice of the viscous term, since the difference of the above two viscous terms is linear and of lower order.
However, as we will see below, the long-time behavior of a solution can be different by the choice of the viscous term due to the effect of a nondissipative part of a solution coming from the convection term.

Since $S^2$ is 2D and simply connected, the Navier--Stokes equations \eqref{E:NS_Intro} are equivalent to the following equation for the vorticity $\omega=\mathrm{rot}\,\mathbf{u}$:
\begin{align} \label{E:Vort}
  \partial_t\omega+\mathbf{u}\cdot\nabla\omega-\nu(\Delta\omega+2\omega) = f, \quad \mathbf{u} = \mathbf{n}_{S^2}\times\nabla\Delta^{-1}\omega \quad\text{on}\quad S^2\times(0,\infty).
\end{align}
Here $f=\mathrm{rot}\,\mathbf{f}$ is an external force, $\mathbf{n}_{S^2}$ is the unit outward normal vector field of $S^2$, and $\Delta$ is Laplace--Beltrami operator on $S^2$ with inverse $\Delta^{-1}$ in $L_0^2(S^2)$, the space of $L^2$ functions on $S^2$ with vanishing mean.
Also, $\mathbf{a}\cdot\mathbf{b}$ and $\mathbf{a}\times\mathbf{b}$ stand for the inner and vector products of $\mathbf{a},\mathbf{b}\in\mathbb{R}^3$.
We refer to \cite{Miu21pre} for the derivation of \eqref{E:Vort} from \eqref{E:NS_Intro}.

For $n\in\mathbb{Z}_{\geq0}$ and $|m|\leq n$, let $Y_n^m$ be the spherical harmonics and $\lambda_n=n(n+1)$ be the corresponding eigenvalue of $-\Delta$ (see Section \ref{S:Pre} for details).
Then, for $n\in\mathbb{N}$ and $a\in\mathbb{R}$, the vorticity equation \eqref{E:Vort} with external force $f=a\nu(\lambda_n-2)Y_n^0$ has a stationary solution with corresponding velocity field
\begin{align} \label{E:Zna_Intro}
  \omega_n^a(\theta,\varphi) = aY_n^0(\theta), \quad \mathbf{u}_n^a(\theta,\varphi) = -\frac{a}{\lambda_n\sin\theta}\frac{dY_n^0}{d\theta}(\theta)\partial_\varphi x(\theta,\varphi),
\end{align}
where $x(\theta,\varphi)$ is the parametrization of $S^2$ by the colatitude $\theta$ and the longitude $\varphi$.
The flow \eqref{E:Zna_Intro} can be seen as a spherical version of the well-known plane Kolmogorov flow, which is a stationary solution to the Navier--Stokes equations in a 2D flat torus (see e.g. \cite{MesSin61,Iud65,Mar86,OkaSho93,MatMiy02} for the study of the stability of the plane Kolmogorov flow).
In \cite{Ili04}, the flow \eqref{E:Zna_Intro} is called the generalized Kolmogorov flow.
Also, it is called an $n$-jet zonal flow in \cite{SaTaYa13,SaTaYa15}.
We call \eqref{E:Zna_Intro} the $n$-jet Kolmogorov type flow in order to emphasize both the similarity to the plane Kolmogorov flow and the number of jets.

In this paper we focus on the case $n=2$ and consider the stability of the two-jet Kolmogorov type flow for the vorticity equation \eqref{E:Vort} (see Remark \ref{R:Onejet} for the one-jet case).
Our particular interest is in studying the effect of a nondissipative part of a perturbation in the stability problem.
Since $-\Delta Y_1^m=2Y_1^m$ for $|m|=0,1$, the viscosity does not work for functions in $\mathrm{span}\{Y_1^0,Y_1^{\pm1}\}$, for which the corresponding velocity fields are Killing vector fields on $S^2$.
If a perturbation contains such a nondissipative part, then it seems to be natural to ask how the nondissipative part affects the long-time behavior of the perturbation through the convection term.
Our aim is to give an explicit answer to this problem.

The stability of the Kolmogorov type flows was studied by Ilyin \cite{Ili04} and Sasaki, Takehiro, and Yamada \cite{SaTaYa13,SaTaYa15} for the Navier--Stokes equations on $S^2$ and by Taylor \cite{Tay16} for the Euler equations on $S^2$ (and these authors dealt with the case of a rotating sphere).
For the viscous case, Ilyin \cite{Ili04} investigated the linear stability and showed that the $n$-jet Kolmogorov type flow is globally asymptotically stable for all $\nu>0$ when $n=1,2$ but it becomes unstable for small $\nu>0$ when $n\geq3$.
In that paper, however, the viscous term in the Navier--Stokes equations was taken as $\nu\Delta_H\mathbf{u}$, which becomes $\nu\Delta\omega$ in the vorticity form.
Hence a perturbation does not contain a nondissipative part in the setting of \cite{Ili04}.
For the case of the viscous term $\nu(\Delta\omega+2\omega)$, Sasaki, Takehiro, and Yamada studied the linear and nonlinear stability of the Kolmogorov type flows in \cite{SaTaYa13} and \cite{SaTaYa15}, respectively, and obtained the same results as in \cite{Ili04}.
However, they considered a perturbation in the orthogonal complement of $\mathrm{span}\{Y_1^0,Y_1^{\pm1}\}$, so the effect of a nondissipative part was not taken into account.

In our previous work \cite{Miu21pre}, we studied the linear stability of the two-jet Kolmogorov type flow under a perturbation which contains a nondissipative part.
The two-jet Kolmogorov type flow is of the form
\begin{align} \label{E:TKol}
  \omega_n^a(\theta,\varphi) = aY_n^0(\theta) = \frac{a}{4}\sqrt{\frac{5}{\pi}}(3\cos^2\theta-1).
\end{align}
Then the linearized equation for the vorticity equation \eqref{E:Vort} around \eqref{E:TKol} is
\begin{align} \label{E:TKol_Linear}
  \partial_t\omega = \nu(\Delta\omega+2\omega)-a_2\cos\theta\,\partial_\varphi(I+6\Delta^{-1})\omega, \quad a_2 = \frac{a}{4}\sqrt{\frac{5}{\pi}} \quad\text{on}\quad S^2\times(0,\infty),
\end{align}
where $I$ is the identity operator.
We refer to \cite{Miu21pre} for the derivation of \eqref{E:TKol_Linear} from \eqref{E:Vort}.
In \cite{Miu21pre} we proved that a solution $\omega(t)$ to \eqref{E:TKol_Linear} with initial data $\omega_0\in L_0^2(S^2)$ satisfies
\begin{align*}
  \|\omega(t)-\omega_{\infty,\mathrm{lin}}\|_{L^2(S^2)} \leq C\left(1+\frac{|a_2|}{\nu}\right)e^{-4\nu t}\|\omega_0\|_{L^2(S^2)}, \quad t\geq0,
\end{align*}
where the equilibrium $\omega_{\infty,\mathrm{lin}}$ is given by
\begin{align} \label{E:Lim_Lin}
  \omega_{\infty,\mathrm{lin}} = \omega_{0,1}^0Y_1^0+\sum_{m=\pm1}\omega_{0,1}^m\left(Y_1^m+\frac{a_2}{\nu}\frac{im}{2\sqrt{5}}Y_2^m\right), \quad \omega_{0,1}^m = (\omega_0,Y_1^m)_{L^2(S^2)}
\end{align}
and $C>0$ is a constant independent of $t$, $\nu$, $a_2$, and $\omega_0$.
Note that $\omega_{\infty,\mathrm{lin}}$ contains the nonzero $Y_2^{\pm1}$-components when $\omega_{0,1}^{\pm1}\neq0$, although the functions $Y_2^{\pm1}$ themselves are dissipative in \eqref{E:TKol_Linear}.
This result shows that the nondissipative part indeed affects the long-time behavior of the perturbation through the interaction between the viscosity and convection even in the linear stability case.
In \cite{Miu21pre} and the companion paper \cite{MaeMiupre} we also observed that a part of a solution to \eqref{E:TKol_Linear} decays at a rate faster than the usual viscous rate $O(e^{-\nu t})$ when $\nu>0$ is small.
Such a phenomenon is called the enhanced dissipation, and it has been attracting interest of many researchers in recent years (see e.g. \cite{CoKiRyZl08,Zla10,Wei21} and \cite{BecWay13,LinXu19,IbMaMa19,WeiZha19,WeZhZh20} for the study of the enhanced dissipation for advection-diffusion equations and for the plane Kolmogorov flow, respectively).

In this paper we study the nonlinear stability of the two-jet Kolmogorov type flow under a perturbation with nondissipative part.
The nonlinear stability problem for the vorticity equation \eqref{E:Vort} around the two-jet Kolmogorov type flow \eqref{E:TKol} is
\begin{align} \label{E:NL_Pert}
  \left\{
  \begin{aligned}
    \partial_t\omega &= \nu(\Delta\omega+2\omega)-\frac{a}{4}\sqrt{\frac{5}{\pi}}\cos\theta\,\partial_\varphi(I+6\Delta^{-1})\omega-\mathbf{u}\cdot\nabla\omega \quad\text{on}\quad S^2\times(0,\infty), \\
    \mathbf{u} &= \mathbf{n}_{S^2}\times\nabla\Delta^{-1}\omega \quad\text{on}\quad S^2\times(0,\infty), \\
    \omega|_{t=0} &= \omega_0 \quad\text{on}\quad S^2,
  \end{aligned}
  \right.
\end{align}
where $\omega_0\in L_0^2(S^2)$ is an initial perturbation (see \cite{Miu21pre} for the derivation of the perturbation operator).
To state our main result, we fix some notations.
For $u\in L_0^2(S^2)$, we write
\begin{align*}
  u_n^m = (u,Y_n^m)_{L^2(S^2)}, \quad u_{=n} = \sum_{m=-n}^nu_n^mY_n^m, \quad u_{\geq N} = \sum_{n\geq N}u_{=n}.
\end{align*}
Note that $u_n^{-m}=(-1)^m\overline{u_n^m}$ if $u$ is real-valued, since $Y_n^{-m}=(-1)^m\overline{Y_n^m}$ (see Section \ref{S:Pre}).
In what follows, we consider real-valued initial data and solutions to \eqref{E:NL_Pert}, although we take the inner product of \eqref{E:NL_Pert} with the complex spherical harmonics $Y_n^m$ in order to make some expressions simple.
For a real-valued initial data $\omega_0\in L_0^2(S^2)$, we set
\begin{align} \label{E:Def_Alb}
  \alpha = \frac{1}{2\sqrt{6\pi}}\,\omega_{0,1}^1\in\mathbb{C}, \quad b = \frac{1}{2\sqrt{3\pi}}\,\omega_{0,1}^0\in\mathbb{R}.
\end{align}
Then we define $\omega_\infty=\sum_{m=-2}^2\omega_{2,\infty}^mY_2^m\in L_0^2(S^2)$ by
\begin{align} \label{E:Def_OmLim}
  \begin{aligned}
    \omega_{2,\infty}^0 &= -12a\cdot\frac{|\alpha|^2(4\nu^2+|\alpha|^2+b^2)}{(4\nu^2+4|\alpha|^2+b^2)(16\nu^2+4|\alpha|^2+b^2)}, \\
    \omega_{2,\infty}^1 &= \frac{\sqrt{6}\,i\alpha(4\nu+2ib)}{4|\alpha|^2+(4\nu+ib)(4\nu+2ib)}(\omega_{2,\infty}^0+a), \quad \omega_{2,\infty}^{-1} = -\overline{\omega_{2,\infty}^1}, \\
    \omega_{2,\infty}^2 &= \frac{2i\alpha}{4\nu+2ib}\,\omega_{2,\infty}^1, \quad \omega_{2,\infty}^{-2} = \overline{\omega_{2,\infty}^2}.
  \end{aligned}
\end{align}
Also, we set $H_0^1(S^2) = L_0^2(S^2)\cap H^1(S^2)$ and write $H_0^{-1}(S^2)$ for the dual space of $H_0^1(S^2)$ with duality product $\langle\cdot,\cdot\rangle_{H_0^1}$.
The main result of this paper is as follows.

\begin{theorem} \label{T:NLS_Main}
  For all real-valued initial data $\omega_0\in L_0^2(S^2)$, there exists a real-valued unique global weak solution
  \begin{align} \label{E:WS_Class}
    \begin{aligned}
      &\omega\in C([0,\infty);L_0^2(S^2))\cap L_{loc}^2([0,\infty);H_0^1(S^2)) \\
      &\text{with} \quad \partial_t\omega\in L_{loc}^2([0,\infty);H_0^{-1}(S^2))
    \end{aligned}
  \end{align}
  to \eqref{E:NL_Pert} in the sense that $\omega(0)=\omega_0$ in $L_0^2(S^2)$ and
  \begin{align} \label{E:Weak_Form}
    \begin{aligned}
      \langle\partial_t\omega(t),\psi\rangle_{H_0^1} &= -\nu(\nabla\omega(t),\nabla\psi)_{L^2(S^2)}+2\nu(\omega(t),\psi)_{L^2(S^2)} \\
      &\qquad -\frac{a}{4}\sqrt{\frac{5}{\pi}}(\cos\theta\,\partial_\varphi(I+6\Delta^{-1})\omega(t),\psi)_{L^2(S^2)} \\
      &\qquad -(\mathbf{u}(t)\cdot\nabla\omega(t),\psi)_{L^2(S^2)}
    \end{aligned}
  \end{align}
  for all real-valued $\psi\in H_0^1(S^2)$ and a.e. $t>0$.
  Moreover,
  \begin{align} \label{E:NP_1and3}
    \omega_{=1}(t) = \omega_{0,=1}, \quad \|\omega_{\geq 3}(t)\|_{L^2(S^2)} \leq e^{-10\nu t}\|\omega_{0,\geq3}\|_{L^2(S^2)}
  \end{align}
  for all $t\geq 0$ and
  \begin{align} \label{E:NP_2}
    \|\omega_{=2}(t)-\omega_\infty\|_{L^2(S^2)} \leq \sigma_1e^{-2\nu t}\left(\|\omega_{0,=2}-\omega_\infty\|_{L^2(S^2)}+\frac{\sigma_2}{\nu}\right)
  \end{align}
  for all $t\geq 0$.
  Here $\omega_\infty=\sum_{m=-2}^2\omega_{2,\infty}^mY_2^m$ is given by \eqref{E:Def_Alb}--\eqref{E:Def_OmLim}.
  Also,
  \begin{align} \label{E:Sigma_12}
    \begin{aligned}
      \sigma_1 &= \exp\left(\frac{C_1}{\nu}\|\omega_{0,\geq3}\|_{L^2(S^2)}\right), \\
      \sigma_2 &= C_2\Bigl(|a|+\|\omega_{0,\geq3}\|_{L^2(S^2)}+\|\omega_\infty\|_{L^2(S^2)}\Bigr)\|\omega_{0,\geq3}\|_{L^2(S^2)},
    \end{aligned}
  \end{align}
  and $C_1,C_2>0$ are constants independent of $t$, $\nu$, $a$, and $\omega_0$.
\end{theorem}

Note that the last term of \eqref{E:Weak_Form} is well-defined since
\begin{align*}
  \mathbf{u} = \mathbf{n}_{S^2}\times\nabla\Delta^{-1}\omega \in H^2(S^2) \subset L^\infty(S^2)
\end{align*}
for $\omega\in H_0^1(S^2)$ by the Sobolev embedding (see e.g. \cite{Aub98}).
Also, a test function $\psi$ in \eqref{E:Weak_Form} is assumed to be real-valued, but in fact we may take a complex-valued $\psi$ by considering its real and imaginary parts separately.

\begin{remark} \label{R:OmLim}
  Compared to the linear stability case (see \eqref{E:Lim_Lin}), the equilibrium $\omega_{0,=1}+\omega_\infty$ is complicated but still determined explicitly.
  Moreover, the $Y_2^0$- and $Y_2^{\pm2}$-components of the equilibrium does not vanish in general for the nonlinear stability problem.
  We also note that the behavior \eqref{E:NP_1and3} of $\omega_{=1}(t)$ and $\omega_{\geq3}(t)$ is the same as in the linear stability case (see \cite[Theorem 3.1]{Miu21pre}) and only the behavior \eqref{E:NP_2} of $\omega_{=2}(t)$ is different.
\end{remark}

\begin{remark} \label{R:Onejet}
  For the one-jet case, we can show that the $Y_1^m$-components of a solution to the nonlinear stability problem is preserved in time as in \eqref{E:NP_1and3} and the $Y_n^m$-components with $n\geq2$ decay exponentially in time.
  The proof is the same as that of Theorem \ref{T:NLS_Main} given in Section \ref{S:NLS} and much easier, so we just give the outline in Section \ref{S:Onejet}.
\end{remark}

It is expected that we can show that the enhanced dissipation occurs for a part of a solution to \eqref{E:NL_Pert} as in the linear stability case \cite{Miu21pre,MaeMiupre}, but we need to analyze carefully the interaction between various components of a solution with different longitudinal wave numbers through the convection term.
The study of the enhanced dissipation for the nonlinear stability problem will be done in another paper.

The result of Theorem \ref{T:NLS_Main} can be extended to the case of a rotating sphere.
In that case, we consider the vorticity equation with Coriolis force (see e.g. \cite{Ili04,SaTaYa13,SaTaYa15})
\begin{align} \label{E:Rotate}
  \left\{
  \begin{aligned}
    &\partial_t\zeta+\mathbf{v}\cdot\nabla\zeta+2\Omega\,\partial_\varphi\Delta^{-1}\zeta-\nu(\Delta\zeta+2\zeta) = f \quad\text{on}\quad S^2\times(0,\infty), \\
    &\mathbf{v} = \mathbf{n}_{S^2}\times\nabla\Delta^{-1}\zeta \quad\text{on}\quad S^2\times(0,\infty),
  \end{aligned}
  \right.
\end{align}
where $\Omega\in\mathbb{R}$ is the rotation speed of a sphere.
For each $n\in\mathbb{N}$ and $a\in\mathbb{R}$, the Kolmogorov type flow \eqref{E:Zna_Intro} is still a stationary solution to \eqref{E:Rotate} with $f=a\nu(\lambda_n-2)Y_n^0$ since it is independent of the longitude $\varphi$.
When $n=2$, the nonlinear stability problem for \eqref{E:Rotate} around the two-jet Kolmogorov type flow \eqref{E:TKol} is
\begin{align} \label{E:NLP_Rot}
  \left\{
  \begin{aligned}
    \partial_t\zeta &= \nu(\Delta\zeta+2\zeta)-2\Omega\,\partial_\varphi\Delta^{-1}\zeta \\
    &\qquad -\frac{a}{4}\sqrt{\frac{5}{\pi}}\cos\theta\,\partial_\varphi(I+6\Delta^{-1})\zeta-\mathbf{v}\cdot\nabla\zeta \quad\text{on}\quad S^2\times(0,\infty), \\
    \mathbf{v} &= \mathbf{n}_{S^2}\times\nabla\Delta^{-1}\zeta \quad\text{on}\quad S^2\times(0,\infty), \\
    \zeta|_{t=0} &= \zeta_0 \quad\text{on}\quad S^2.
  \end{aligned}
  \right.
\end{align}
Then we easily find by direct calculations (see Section \ref{S:Rotate} for the outline) that solutions $\omega$ to \eqref{E:NL_Pert} and $\zeta$ to \eqref{E:NLP_Rot} are related by
\begin{align} \label{E:Sol_CoV}
  \omega(\theta,\varphi,t) = \zeta(\theta,\varphi-\Omega t,t)+2\Omega\cos\theta = \zeta(\theta,\varphi-\Omega t,t)+4\sqrt{\frac{\pi}{3}}\,\Omega Y_1^0(\theta)
\end{align}
in spherical coordinates.
Thus, for a given $\zeta_0\in L_0^2(S^2)$, we apply Theorem \ref{T:NLS_Main} to
\begin{align*}
  \omega_0(\theta,\varphi) = \zeta_0(\theta,\varphi)+2\Omega\cos\theta = \zeta_0(\theta,\varphi)+4\sqrt{\frac{\pi}{3}}\,\Omega Y_1^0(\theta),
\end{align*}
substitute \eqref{E:Sol_CoV} for \eqref{E:NP_1and3} and \eqref{E:NP_2}, make the change of variable $\varphi\mapsto\varphi+\Omega t$ which does not change the $L^2(S^2)$-norm, and use $Y_n^m(\theta,\varphi+\Omega t)=e^{im\Omega t}Y_n^m(\theta,\varphi)$ to obtain the following result.

\begin{theorem} \label{T:Rotate}
  For all real-valued initial data $\zeta_0\in L_0^2(S^2)$, there exists a unique global weak solution $\zeta$ to \eqref{E:NLP_Rot} in the class \eqref{E:WS_Class} and
  \begin{align*}
    \zeta_{=1}(t) = \sum_{m=0,\pm1}e^{im\Omega t}\zeta_{0,1}^mY_1^m, \quad \|\zeta_{\geq3}(t)\| \leq e^{-10\nu t}\|\zeta_{0,\geq3}\|_{L^2(S^2)}
  \end{align*}
  for all $t\geq0$.
  Moreover, if we set
  \begin{align*}
    \alpha = \frac{1}{2\sqrt{6\pi}}\,\zeta_{0,1}^1\in\mathbb{C}, \quad b_\Omega = \frac{1}{2\sqrt{3\pi}}\,\zeta_{0,1}^0+\frac{2}{3}\,\Omega \in \mathbb{R}
  \end{align*}
  and define $\zeta_\infty(t)=\sum_{m=-2}^2e^{im\Omega t}\zeta_{2,\infty}^mY_2^m$ by
  \begin{align*}
    \zeta_{2,\infty}^0 &= -12a\cdot\frac{|\alpha|^2(4\nu^2+|\alpha|^2+b_\Omega^2)}{(4\nu^2+4|\alpha|^2+b_\Omega^2)(16\nu^2+4|\alpha|^2+b_\Omega^2)}, \\
    \zeta_{2,\infty}^1 &= \frac{\sqrt{6}\,i\alpha(4\nu+2ib_\Omega)}{4|\alpha|^2+(4\nu+ib_\Omega)(4\nu+2ib_\Omega)}(\zeta_{2,\infty}^0+a), \quad \zeta_{2,\infty}^{-1} = -\overline{\zeta_{2,\infty}^1}, \\
    \zeta_{2,\infty}^2 &= \frac{2i\alpha}{4\nu+2ib_\Omega}\,\zeta_{2,\infty}^1, \quad \zeta_{2,\infty}^{-2} = \overline{\zeta_{2,\infty}^2},
  \end{align*}
  then we have
  \begin{align*}
    \|\zeta_{=2}(t)-\zeta_\infty(t)\|_{L^2(S^2)} \leq \sigma_1\left(\|\zeta_{0,=2}-\zeta_\infty(0)\|_{L^2(S^2)}+\frac{\sigma_2}{\nu}\right)
  \end{align*}
  for all $t\geq 0$, where $\sigma_1$ and $\sigma_2$ are given by \eqref{E:Sigma_12} with $\omega_{0,\geq3}$ and $\omega_\infty$ replaced by $\zeta_{0,\geq3}$ and $\zeta_\infty(0)$, respectively, and with constants $C_1,C_2>0$ independent of $t$, $\nu$, $a$, $\Omega$, and $\zeta_0$.
\end{theorem}

Let us explain the outline of the proof of Theorem \ref{T:NLS_Main}.
Since $S^2$ is 2D, we can show the global existence and uniqueness of a weak solution to \eqref{E:NL_Pert} by the Galerkin method with basis functions $Y_n^m$ and the energy method as in the case of the Navier--Stokes equations in 2D flat domains (see e.g. \cite{Tem01,BoyFab13}).
Hence we omit details in this paper.
To prove \eqref{E:NP_1and3} and \eqref{E:NP_2}, we first take the inner product of \eqref{E:NL_Pert} with $Y_1^m$, $|m|=0,1$.
Then we find that $\frac{d}{dt}\omega_1^m(t)=0$ and thus $\omega_{=1}(t)=\omega_{0,=1}$ by using $-\Delta Y_n^m=\lambda_nY_n^m$ with $\lambda_1=2$ and $\lambda_2=6$, a recurrence relation for $\cos\theta\,Y_n^m$ (see \eqref{E:Y_Rec}), and the fact that a vector field of the form
\begin{align*}
  \mathbf{n}_{S^2}\times\nabla\chi, \quad \chi = \cos\theta, \sin\theta\cos\varphi, \sin\theta\sin\varphi
\end{align*}
is a Killing vector field on $S^2$.
Here the last fact is essential to show that the inner product of the convection term with $Y_1^m$ vanishes (see Lemma \ref{L:S2Int_Y1m}).
To prove that result, we apply an identity for a Killing vector field on a general manifold given in Lemma \ref{L:Killing}, which seems to have its own interest.

Next we derive an equation for $\omega_{\geq2}(t)=\omega(t)-\omega_{0,=1}$ from \eqref{E:NL_Pert} and take the inner product of that equation with $(I+6\Delta^{-1})\omega_{\geq2}(t)$.
Then we obtain a differential inequality which yields the estimate \eqref{E:NP_1and3} for $\omega_{\geq3}(t)$ by applying $\lambda_2=6$ and $\lambda_n\geq12$ for $n\geq3$, integration by parts, and the identity
\begin{align} \label{E:Intro_XgOm}
  \bigl(X\cdot\nabla\Delta^{-1}\omega_{\geq2}(t),\omega_{\geq2}(t)\bigr)_{L^2(S^2)} = \bigl(X\cdot\nabla\omega_{\geq2}(t),\Delta^{-1}\omega_{\geq2}(t)\bigr)_{L^2(S^2)} = 0
\end{align}
given in Lemma \ref{L:S2Int_Kil}, where $X$ is given by
\begin{align} \label{E:Intro_DefX}
  X = \mathbf{n}_{S^2}\times\nabla\Delta^{-1}\omega_{0,=1} = -\frac{1}{2}\mathbf{n}_{S^2}\times\nabla\omega_{0,=1}.
\end{align}
Here we do not have an estimate for $\omega_{\geq2}(t)=\omega_{=2}(t)+\omega_{\geq3}(t)$ since $(I+6\Delta^{-1})\omega_{=2}(t)=0$ by $\lambda_2=6$.
Moreover, the identity \eqref{E:Intro_XgOm} is used to show
\begin{align*}
  \bigl(X\cdot\nabla(I+2\Delta^{-1})\omega_{\geq2}(t),(I+6\Delta^{-1})\omega_{\geq2}(t)\bigr)_{L^2(S^2)} = 0,
\end{align*}
where the term $X\cdot\nabla(I+2\Delta^{-1})\omega_{\geq2}(t)$ comes from the interaction between the nondissipative part $\omega_{0,=1}$ and the dissipative part $\omega_{\geq2}(t)$ through the convection term.
We further note that, since the above $X$ is of the form $X(x)=\mathbf{a}\times x$, $x\in S^2$ with some $\mathbf{a}\in\mathbb{R}^3$ expressed in terms of $\omega_{0,1}^0$ and $\omega_{0,1}^1$, it is a Killing vector field on $S^2$ and thus we can apply again the identity for a Killing vector field on a general manifold given in Lemma \ref{L:Killing} to get the identity \eqref{E:Intro_XgOm}.

Lastly, we derive an equation for $\omega_{=2}(t)=\omega_{\geq2}(t)-\omega_{\geq3}(t)$, prove the estimate \eqref{E:NP_2}, and determine $\omega_\infty=\sum_{m=-2}^2\omega_{2,\infty}^mY_2^m$.
It turns out that the evolution of $\omega_{=2}(t)$ is described by a system of linear ordinary differential equations (ODEs) of the form
\begin{align} \label{E:Intro_ODE}
  \begin{aligned}
    \frac{d}{dt}\bm{\omega}(t) &= -\{4\nu\bm{I}_5+i\bm{A}+\bm{M}(t)\}\bm{\omega}(t)+\bm{f}_{\geq3}(t)+\bm{c}, \quad t>0, \\
    \bm{\omega}(t) &= \bigl(\omega_2^2(t),\dots,\omega_2^{-2}(t)\bigr)^T \in \mathbb{C}^5,
  \end{aligned}
\end{align}
where $\bm{I}_5$ is the $5\times 5$ identity matrix, $\bm{A}$ is a constant self-adjoint matrix, $\bm{M}(t)$ and $\bm{f}_{\geq3}(t)$ are a matrix-valued function and a vector field decaying exponentially in time, and $\bm{c}$ is a constant vector coming from the nondissipative part $\omega_{0,=1}$ through the perturbation operator (see \eqref{E:Def_Amk} and \eqref{E:ODE_Nota} for the precise definitions).
Then, noting that $4\nu\bm{I}_5+i\bm{A}$ is invertible since $\bm{A}$ is self-adjoint, we set
\begin{align*}
  \bm{\omega}_\infty = (\omega_{2,\infty}^2,\dots\omega_{2,\infty}^{-2})^T = (4\nu\bm{I}_5+i\bm{A})^{-1}\mathbf{c}
\end{align*}
and use \eqref{E:Intro_ODE} to derive an estimate for $\bm{\omega}(t)-\bm{\omega}_\infty$ which corresponds to \eqref{E:NP_2}.
Moreover, using the explicit forms of $\bm{A}$ and $\bm{c}$, we can solve $(4\nu\bm{I}_5+i\bm{A})\bm{\omega}_\infty=\bm{c}$ to determine $\bm{\omega}_\infty$.
Here the matrix $\bm{A}$ comes from the interaction between the nondissipative part $\omega_{0,=1}$ and the dissipative part $\omega_{=2}(t)$ through the convection term.
Indeed, each entry of $\bm{A}$ is given by the $Y_2^m$-component of $X\cdot\nabla\omega_{=2}$ with $X$ given by \eqref{E:Intro_DefX}.
Also, we can express $X\cdot\nabla Y_2^m$ as a linear combination of $Y_2^0$, $Y_2^{\pm1}$, and $Y_2^{\pm2}$ for each $|m|=0,1,2$ (see Lemma \ref{L:Y2m_Kil}), so we can write $\bm{A}$ explicitly in terms of $\omega_{0,1}^0$ and $\omega_{0,1}^1$ and find that $\bm{A}$ is self-adjoint.

The rest of this paper is organized as follows.
In Section \ref{S:Pre} we fix notations and give auxiliary results on calculus on a manifold and $S^2$.
The main part of this paper is Section \ref{S:NLS}, which is devoted to the proof of Theorem \ref{T:NLS_Main}.
In Section \ref{S:PoYK} we give the proof of Lemma \ref{L:Y2m_Kil} which consists of elementary but slightly long calculations.
As appendices, we briefly explain the behavior of a perturbation for the one-jet Kolmogorov type flow in Section \ref{S:Onejet} and observe that solutions to \eqref{E:NL_Pert} and \eqref{E:NLP_Rot} are related by \eqref{E:Sol_CoV} in Section \ref{S:Rotate}.

\section{Preliminaries} \label{S:Pre}
We fix notations and give auxiliary results on calculus on a manifold and $S^2$.

\subsection{Calculus on a manifold} \label{SS:Manifold}
For $n\geq2$ let $M$ be an $n$-dimensional Riemannian manifold without boundary.
Note that $M$ is a real manifold, and in this subsection we only consider real-valued functions and vector fields on $M$.
Let $\langle\cdot,\cdot\rangle$, $\nabla$, and $d\mathcal{H}^n$ be the Riemannian metric, the Levi-Civita connection, and the volume form on $M$.
For a function $f$ and a vector field $X$ on $M$, we write $\nabla f$, $\mathrm{div}\,X$, and $\nabla_Xf=\langle X,\nabla f\rangle$ for the gradient of $f$, the divergence of $X$, and the directional derivative of $f$ along $X$.
We also denote by $\Delta=\mathrm{div}\,\nabla$ the Laplace--Beltrami operator on $M$.
Note that
\begin{align} \label{E:Compat}
  \nabla_X\langle Y,Z\rangle = \langle\nabla_XY,Z\rangle+\langle Y,\nabla_XZ\rangle
\end{align}
for vector fields $X,Y,Z$ on $M$, since $\nabla$ is compatible with $\langle\cdot,\cdot\rangle$.
For a function $f$ on $M$, let $\nabla^2f$ be the covariant Hessian of $f$ given by
\begin{align} \label{E:Def_Cov}
  (\nabla^2f)(Y,X) = \nabla_X(\nabla_Yf)-\nabla_{\nabla_XY}f = \nabla_X\langle Y,\nabla f\rangle-\langle\nabla_XY,\nabla f\rangle
\end{align}
for vector fields $X,Y$ on $M$ (see \cite{Lee18}).
As in the flat space case, $\nabla^2f$ is symmetric in the sense that $(\nabla^2f)(Y,X)=(\nabla^2f)(X,Y)$.
A (smooth) vector field $X$ on $M$ is called Killing if $\langle\nabla_YX,Z\rangle+\langle Y,\nabla_ZX\rangle=0$ for all vector fields $Y$ and $Z$ on $M$.
Note that $\mathrm{div}\,X=0$ on $M$ if $X$ is Killing, since $\mathrm{div}\,X=\sum_{i=1}^n\langle\nabla_{\tau_i}X,\tau_i\rangle$ at each $x\in M$, where $\{\tau_1,\dots,\tau_n\}$ is an orthonormal basis of the tangent plane of $M$ at $x$.
For $k\geq0$ we denote by $H^k(M)$ the Sobolev spaces of $L^2$ functions on $M$ (see e.g. \cite{Aub98}).

\begin{lemma} \label{L:Killing}
  Let $X$ be a Killing vector field on $M$.
  Then
  \begin{align} \label{E:Killing}
    \int_M\{(\Delta f)\langle\nabla g,X\rangle+(\Delta g)\langle\nabla f,X\rangle\}\,d\mathcal{H}^n = 0
  \end{align}
  for all real-valued functions $f,g\in H^2(M)$.
\end{lemma}

\begin{proof}
  Let $F=\nabla_{\nabla f}\langle\nabla g,X\rangle+\nabla_{\nabla g}\langle\nabla f,X\rangle$ on $M$.
  Then
  \begin{align*}
    \int_M\{(\Delta f)\langle\nabla g,X\rangle+(\Delta g)\langle\nabla f,X\rangle\}\,d\mathcal{H}^n = -\int_MF\,d\mathcal{H}^n
  \end{align*}
  by integration by parts.
  Moreover, by \eqref{E:Def_Cov} and the symmetry of $\langle\cdot,\cdot\rangle$, we have
  \begin{align*}
    &(\nabla^2f)(X,\nabla g)+(\nabla^2g)(X,\nabla f) \\
    &\qquad = \nabla_{\nabla g}\langle X,\nabla f\rangle-\langle\nabla_{\nabla g}X,\nabla f\rangle+\nabla_{\nabla f}\langle X,\nabla g\rangle-\langle\nabla_{\nabla f}X,\nabla g\rangle \\
    &\qquad = F-(\langle\nabla_{\nabla f}X,\nabla g\rangle+\langle\nabla f,\nabla_{\nabla g}X\rangle)
  \end{align*}
  and
  \begin{align*}
    &(\nabla^2f)(\nabla g,X)+(\nabla^2g)(\nabla f,X) \\
    &\qquad = \nabla_X\langle\nabla g,\nabla f\rangle-\langle\nabla_X\nabla g,\nabla f\rangle+\nabla_X\langle\nabla f,\nabla g\rangle-\langle\nabla_X\nabla f,\nabla g\rangle \\
    &\qquad = \nabla_X\langle\nabla f,\nabla g\rangle
  \end{align*}
  on $M$.
  In the last equality, we also used \eqref{E:Compat} with $Y=\nabla f$ and $Z=\nabla g$.
  Since $\nabla^2 f$ and $\nabla^2 g$ are symmetric, it follows from the above equalities that
  \begin{align*}
    F-(\langle\nabla_{\nabla f}X,\nabla g\rangle+\langle\nabla f,\nabla_{\nabla g}X\rangle) = \nabla_X\langle\nabla f,\nabla g\rangle
  \end{align*}
  on $M$.
  Moreover, $\langle\nabla_{\nabla f}X,\nabla g\rangle+\langle\nabla f,\nabla_{\nabla g}X\rangle=0$ on $M$ since $X$ is Killing.
  Hence
  \begin{align*}
    \int_M\{(\Delta f)\langle\nabla g,X\rangle+(\Delta g)\langle\nabla f,X\rangle\}\,d\mathcal{H}^n &= -\int_MF\,d\mathcal{H}^n = -\int_M\nabla_X\langle\nabla f,\nabla g\rangle\,d\mathcal{H}^n \\
    &= \int_M(\mathrm{div}\,X)\langle\nabla f,\nabla g\rangle\,d\mathcal{H}^n = 0
  \end{align*}
  by integration by parts and $\mathrm{div}\,X=0$ on $M$ since $X$ is Killing.
\end{proof}

\subsection{Calculus on the unit sphere} \label{SS:Sphere}
Now let $M=S^2$ be the 2D unit sphere in $\mathbb{R}^3$ equipped with the Riemannian metric induced by the Euclidean metric of $\mathbb{R}^3$.
We denote by $\mathbf{a}\cdot\mathbf{b}$ and $\mathbf{a}\times\mathbf{b}$ the inner and vector products of $\mathbf{a},\mathbf{b}\in\mathbb{R}^3$.
Let $\mathbf{n}_{S^2}$ be the unit outward normal vector field of $S^2$.
For real tangential vector fields $X$ and $Y$ on $S^2$, the covariant derivative of $X$ along $Y$ is given by
\begin{align*}
  \nabla_YX = (Y\cdot\nabla^{\mathbb{R}^3})\widetilde{X}-\left(\left[(Y\cdot\nabla^{\mathbb{R}^3})\widetilde{X}\right]\cdot\mathbf{n}_{S^2}\right)\mathbf{n}_{S^2} \quad\text{on}\quad S^2,
\end{align*}
where $\nabla^{\mathbb{R}^3}$ is the standard gradient in $\mathbb{R}^3$ and $\tilde{X}$ is an extension of $X$ to an open neighborhood of $S^2$.
Note that the value of $\nabla_YX$ is independent of the choice of $\widetilde{X}$.
Then we easily see that for any $\mathbf{a}\in\mathbb{R}^3$ a real vector field $X(x)=\mathbf{a}\times x$, $x\in S^2$ is tangential and satisfies $\nabla_YX\cdot Z+Y\cdot\nabla_ZX=0$ on $S^2$ for all real tangential vector fields $Y$ and $Z$ on $S^2$, i.e. $X(x)=\mathbf{a}\times x$ is Killing on $S^2$.
We use this fact without mention in the sequel.

In what follows, we mainly consider real-valued functions on $S^2$, but we take the $L^2(S^2)$-inner product of functions with the complex spherical harmonics.
Thus we write
\begin{align*}
  (u,v)_{L^2(S^2)} = \int_{S^2}u\bar{v}\,d\mathcal{H}^2, \quad \|u\|_{L^2(S^2)} = (u,u)_{L^2(S^2)}^{1/2}
\end{align*}
for complex-valued functions $u,v\in L^2(S^2)$, where $\bar{v}$ is the complex conjugate of $v$.
Also, we sometimes abuse the notations of the inner and vector products to write
\begin{align} \label{E:Abu_Pr}
  \mathbf{a}\cdot(\mathbf{b}_1+i\mathbf{b}_2) = \mathbf{a}\cdot\mathbf{b}_1+i\mathbf{a}\cdot\mathbf{b}_2, \quad \mathbf{a}\times(\mathbf{b}_1+i\mathbf{b}_2) = \mathbf{a}\times\mathbf{b}_1+i\mathbf{a}\times\mathbf{b}_2
\end{align}
for $\mathbf{a},\mathbf{b}_1,\mathbf{b}_2\in\mathbb{R}^3$.
We do not encounter the case where $\mathbf{a}$ is complex in the sequel.

Let $\theta$ and $\varphi$ be the colatitude and longitude so that $S^2$ is parametrized by
\begin{align} \label{E:Sphrical}
  x(\theta,\varphi) = (\sin\theta\cos\varphi,\sin\theta\sin\varphi,\cos\theta), \quad \theta\in[0,\pi], \, \varphi\in[0,2\pi).
\end{align}
For a function $u$ on $S^2$, we abuse the notation $u(\theta,\varphi)=(u\circ x)(\theta,\varphi)$ so that
\begin{align} \label{E:Grad}
  \nabla u(\theta,\varphi) = \partial_\theta u(\theta,\varphi)\partial_\theta x(\theta,\varphi)+\frac{\partial_\varphi u(\theta,\varphi)}{\sin^2\theta}\,\partial_\varphi x(\theta,\varphi)
\end{align}
for the gradient of $u$.
Let $Y_n^m$ be the spherical harmonics of the form
\begin{align*}
  Y_n^m = Y_n^m(\theta,\varphi) = \sqrt{\frac{2n+1}{4\pi}\frac{(n-m)!}{(n+m)!}} \, P_n^m(\cos\theta)e^{im\varphi}, \quad n\in\mathbb{Z}_{\geq0}, \, |m|\leq n.
\end{align*}
Here $P_n^m$ are the associated Legendre functions given by
\begin{align*}
  P_n^0(s) &= \frac{1}{2^nn!}\frac{d^n}{ds^n}(s^2-1)^n, \\
  P_n^m(s) &=
  \begin{cases}
    (-1)^m(1-s^2)^{m/2}\displaystyle\frac{d^m}{ds^m}P_n^0(s), &m\geq0, \\
    (-1)^{|m|}\displaystyle\frac{(n-|m|)!}{(n+|m|)!}P_n^{|m|}(s), &m = -|m| < 0
  \end{cases}
\end{align*}
for $s\in(-1,1)$ (see \cite{Led72,DLMF}).
Note that $Y_n^{-m}=(-1)^m\overline{Y_n^m}$ by the above definitions.
It is known (see e.g. \cite{VaMoKh88,Tes14}) that $Y_n^m$ are the eigenfunctions of $-\Delta$ associated with the eigenvalue $\lambda_n=n(n+1)$ for each $n\geq0$, and the set of all $Y_n^m$ forms an orthonormal basis of $L^2(S^2)$.
Hence each $u\in L^2(S^2)$ can be expanded as
\begin{align*}
  u = \sum_{n=0}^\infty\sum_{m=-n}^nu_n^mY_n^m, \quad u_n^m = (u,Y_n^m)_{L^2(S^2)}.
\end{align*}
Moreover, if $u$ is real-valued, then $u_n^{-m}=(-1)^m\overline{u_n^m}$ since $Y_n^{-m}=(-1)^m\overline{Y_n^m}$.
It is also known that the recurrence relation
\begin{align*}
  (n-m+1)P_{n+1}^m(s)-(2n+1)sP_n^m(s)+(n+m)P_{n-1}^m(s) = 0
\end{align*}
holds (see \cite[(7.12.12)]{Led72}) and thus (see also \cite[Section 5.7]{VaMoKh88})
\begin{align} \label{E:Y_Rec}
  \cos\theta\,Y_n^m = a_n^mY_{n-1}^m+a_{n+1}^mY_{n+1}^m, \quad a_n^m = \sqrt{\frac{(n-m)(n+m)}{(2n-1)(2n+1)}}
\end{align}
for $n\in\mathbb{Z}_{\geq0}$ and $|m|\leq n$, where we consider $Y_{|m|-1}^m\equiv0$.

Let $L_0^2(S^2)$ be the space of $L^2$ functions on $S^2$ with vanishing mean, i.e.
\begin{align*}
  L_0^2(S^2) = \left\{u\in L^2(S^2) ~\middle|~ \int_{S^2}u\,d\mathcal{H}^2 = 0\right\} = \{u\in L^2(S^2) \mid (u,Y_0^0)_{L^2(S^2)}=0\}.
\end{align*}
Then $\Delta$ is invertible and self-adjoint as a linear operator
\begin{align*}
  \Delta\colon D_{L_0^2(S^2)}(\Delta) \subset L_0^2(S^2) \to L_0^2(S^2), \quad D_{L_0^2(S^2)}(\Delta) = L_0^2(S^2)\cap H^2(S^2).
\end{align*}
Also, the fractional Laplace--Beltrami operator $(-\Delta)^s$ with $s\in\mathbb{R}$ is defined by
\begin{align} \label{E:Def_Laps}
  (-\Delta)^su = \sum_{n=1}^\infty\sum_{m=-n}^n\lambda_n^s(u,Y_n^m)_{L^2(S^2)}Y_n^m, \quad u \in L_0^2(S^2).
\end{align}
Note that $(-\Delta)^su$ is real-valued if $u$ is so.
Moreover,
\begin{align} \label{E:Y_Perturb}
  \cos\theta\,\partial_\varphi(I+6\Delta^{-1})Y_n^m = im\left(1-\frac{6}{\lambda_n}\right)(a_n^mY_{n-1}^m+a_{n+1}^mY_{n+1}^m)
\end{align}
by \eqref{E:Y_Rec}, \eqref{E:Def_Laps}, and $\partial_\varphi Y_n^m=imY_n^m$.
We also have
\begin{align} \label{E:Lap_Half}
  \|(-\Delta)^{1/2}u\|_{L^2(S^2)} = \|\nabla u\|_{L^2(S^2)}, \quad u\in L_0^2(S^2)\cap H^1(S^2)
\end{align}
by a density argument and integration by parts.
Let us give auxiliary results.

\begin{lemma} \label{L:S2Int_Y1m}
  Let $\psi\in H^2(S^2)$ be a real-valued function and $\mathbf{u}=\mathbf{n}_{S^2}\times\nabla\psi$ on $S^2$.
  Then
  \begin{align} \label{E:S2Int_Y1m}
    (\mathbf{u}\cdot\nabla\Delta\psi,Y_1^m)_{L^2(S^2)} = 0, \quad |m|=0,1.
  \end{align}
\end{lemma}

\begin{proof}
  Since
  \begin{align*}
    Y_1^0 = C_1^0\cos\theta, \quad Y_1^{\pm1} = C_1^{\pm1}e^{\pm i\varphi}\sin\theta=C_1^{\pm1}(\sin\theta\cos\varphi\pm i\sin\theta\sin\varphi)
  \end{align*}
  with constants $C_1^0,C_1^{\pm1}\in\mathbb{R}$, it suffices to show that
  \begin{align*}
    (\mathbf{u}\cdot\nabla\Delta\psi,\chi)_{L^2(S^2)} = \int_{S^2}(\mathbf{u}\cdot\nabla\Delta\psi)\chi\,d\mathcal{H}^2 = 0, \quad \chi = \cos\theta,\sin\theta\cos\varphi,\sin\theta\sin\varphi.
  \end{align*}
  Noting that $\mathbf{u}=\mathbf{n}_{S^2}\times\nabla\psi$ is divergence free on $S^2$, we have
  \begin{align*}
    (\mathbf{u}\cdot\nabla\Delta\psi,\chi)_{L^2(S^2)} = -\int_{S^2}(\Delta\psi)(\mathbf{u}\cdot\nabla\chi)\,d\mathcal{H}^2 = \int_{S^2}(\Delta\psi)\{\nabla\psi\cdot(\mathbf{n}_{S^2}\times\nabla\chi)\}\,d\mathcal{H}^2
  \end{align*}
  by integration by parts and $(\mathbf{a}\times\mathbf{b})\cdot\mathbf{c}=-\mathbf{b}\cdot(\mathbf{a}\times\mathbf{c})$ for $\mathbf{a},\mathbf{b},\mathbf{c}\in\mathbb{R}^3$.
  Moreover, we observe by $\mathbf{n}_{S^2}(x)=x$ for $x\in S^2$, \eqref{E:Sphrical}, and \eqref{E:Grad} that
  \begin{align} \label{Pf_SIY:n_gc}
    X(x) = \mathbf{n}_{S^2}(x)\times\nabla \chi(x) =
    \begin{cases}
      -\mathbf{e}_3\times x &\text{if} \quad \chi = \cos\theta, \\
      -\mathbf{e}_1\times x &\text{if} \quad \chi = \sin\theta\cos\varphi, \\
      -\mathbf{e}_2\times x &\text{if} \quad \chi = \sin\theta\sin\varphi
    \end{cases}
  \end{align}
  for $x\in S^2$, where $\{\mathbf{e}_1,\mathbf{e}_2,\mathbf{e}_3\}$ is the standard basis of $\mathbb{R}^3$.
  Since this $X$ is Killing and $\psi$ is real-valued, we can apply \eqref{E:Killing} with $f=g=\psi$ to find that
  \begin{align*}
    (\mathbf{u}\cdot\nabla\Delta\psi,\chi)_{L^2(S^2)} = \int_{S^2}(\Delta\psi)(\nabla\psi\cdot X)\,d\mathcal{H}^2 = 0
  \end{align*}
  and thus \eqref{E:S2Int_Y1m} follows.
\end{proof}

\begin{lemma} \label{L:S2Int_Kil}
  For $\mathbf{a}\in\mathbb{R}^3$ let $X(x)=\mathbf{a}\times x$, $x\in S^2$.
  Then
  \begin{align} \label{E:S2Int_Kil}
    (X\cdot\nabla\Delta^{-1}\omega,\omega)_{L^2(S^2)} = (X\cdot\nabla\omega,\Delta^{-1}\omega)_{L^2(S^2)} = 0
  \end{align}
  for every real-valued function $\omega\in L_0^2(S^2)$.
\end{lemma}

\begin{proof}
  Let $\psi=\Delta^{-1}\omega\in H^2(S^2)$.
  Then since $\psi$ is real-valued and $X$ is a Killing vector field on $S^2$, it follows from \eqref{E:Killing} with $f=g=\psi$ that
  \begin{align*}
    (X\cdot\nabla\Delta^{-1}\omega,\omega)_{L^2(S^2)} = (X\cdot\nabla\psi,\Delta\psi)_{L^2(S^2)} = 0.
  \end{align*}
  Also, noting that $\mathrm{div}\,X=0$ on $S^2$ since $X$ is Killing, we have
  \begin{align*}
    (X\cdot\nabla\omega,\Delta^{-1}\omega)_{L^2(S^2)} = -(X\cdot\nabla\Delta^{-1}\omega,\omega)_{L^2(S^2)} = 0
  \end{align*}
  by integration by parts and the above equality,
  Hence \eqref{E:S2Int_Kil} is valid.
\end{proof}

\begin{lemma} \label{L:Y2m_Kil}
  For $\mathbf{a}=(a_1,a_2,a_3)^T\in\mathbb{R}^3$ let $X(x)=\mathbf{a}\times x$, $x\in S^2$.
  Then
  \begin{align} \label{E:Y2m_Kil}
    \begin{aligned}
      X\cdot\nabla Y_2^0 &= \frac{\sqrt{6}}{2}(ia_1+a_2)Y_2^1+\frac{\sqrt{6}}{2}(ia_1-a_2)Y_2^{-1}, \\
      X\cdot\nabla Y_2^1 &= (ia_1+a_2)Y_2^2+ia_3Y_2^1+\frac{\sqrt{6}}{2}(ia_1-a_2)Y_2^0, \\
      X\cdot\nabla Y_2^2 &= 2ia_3Y_2^2+(ia_1-a_2)Y_2^1.
    \end{aligned}
  \end{align}
  Moreover, since $Y_2^{-m}=(-1)^m\overline{Y_2^m}$ and $X$ is real,
  \begin{align} \label{E:Y2mK_minus}
    \begin{aligned}
      X\cdot\nabla Y_2^{-1} &= (ia_1-a_2)Y_2^{-2}-ia_3Y_2^{-1}+\frac{\sqrt{6}}{2}(ia_1+a_2)Y_2^0, \\
      X\cdot\nabla Y_2^{-2} &= -2ia_3Y_2^{-2}+(ia_1+a_2)Y_2^{-1}.
    \end{aligned}
  \end{align}
\end{lemma}

The equalities \eqref{E:Y2m_Kil} are shown by elementary calculations under spherical coordinates.
To avoid making this section too long, we give the proof of \eqref{E:Y2m_Kil} in Section \ref{S:PoYK}.

\section{Proof of Theorem \ref{T:NLS_Main}} \label{S:NLS}
The purpose of this section is to establish Theorem \ref{T:NLS_Main}.

First we note that the global existence and uniqueness of a weak solution to \eqref{E:NL_Pert} are proved in the same way as in the case of the Navier--Stokes equations in flat 2D domains (see e.g. \cite{Tem01,BoyFab13}).
We can show the global existence of a weak solution by a standard Galerkin method with basis functions $Y_n^m$.
Also, since $S^2$ is 2D, we can get the uniqueness of a weak solution by estimating the difference of two weak solutions with the aid of the weak form \eqref{E:Weak_Form} and Ladyzhenskaya's inequality
\begin{align*}
  \|\omega\|_{L^4(S^2)} \leq \|\omega\|_{L^2(S^2)}^{1/2}\|\omega\|_{H^1(S^2)}^{1/2}, \quad \omega\in H^1(S^2),
\end{align*}
which follows from the same inequality on $\mathbb{R}^2$ (see \cite{Lad69}) and a localization argument with a partition of unity, and then by using Gronwall's inequality.
Here we omit details and just give a remark: approximate solutions constructed by the Galerkin method can grow exponentially in time due to the perturbation operator in \eqref{E:NL_Pert}, but it does not matter since we take a limit of the approximate solutions on finite time intervals, e.g. on $[0,n]$, $n\in\mathbb{N}$ to get weak solutions $\omega_n$ on $[0,n]$ and then we use the uniqueness of a weak solution to define a global weak solution $\omega$ by $\omega=\omega_n$ on $[0,n]$ for each $n\in\mathbb{N}$.
We also refer to \cite{Ski17} for the proof in the case of a modified vorticity equation on $S^2$.

Now let $\omega$ be the unique global weak solution to \eqref{E:NL_Pert} with initial data $\omega_0\in L_0^2(S^2)$.
Here we assume that $\omega_0$ is real-valued and thus $\omega(t)$ is so.
We write
\begin{align*}
  \omega_n^m(t) = (\omega(t),Y_n^m)_{L^2(S^2)}, \quad \omega_{=n}(t) = \sum_{m=-n}^n\omega_n^m(t)Y_n^m, \quad \omega_{\geq N}(t) = \sum_{n\geq N}\omega_{=n}(t)
\end{align*}
and similarly for $\omega_0$ so that
\begin{align*}
  \omega(t) = \sum_{n=1}^\infty\sum_{m=-n}^n\omega_n^m(t)Y_n^m, \quad \omega_0 = \sum_{n=1}^\infty\sum_{m=-n}^n\omega_{0,n}^mY_n^m.
\end{align*}
Then, since $\omega(t)$ and $\omega_0$ are real-valued and $Y_n^{-m}=(-1)^m\overline{Y_n^m}$,
\begin{align} \label{E:Om_Conj}
  \omega_n^{-m}(t) = (-1)^m\overline{\omega_n^m(t)}, \quad \omega_{0,n}^{-m} = (-1)^m\overline{\omega_{0,n}^m}
\end{align}
and thus $\omega_{=n}(t)$, $\omega_{\geq N}(t)$, $\omega_{0,=n}$, and $\omega_{0,\geq N}$ are real-valued.

Let us show \eqref{E:NP_1and3} and \eqref{E:NP_2}.
The proof consists of six steps.
In what follows, for the sake of simplicity, we say that we take the inner product of an equation like \eqref{E:NL_Pert} with a test function instead of saying that we use a corresponding weak form like \eqref{E:Weak_Form}.
Also, we frequently use the following fact without mention: a vector field of the form $\mathbf{v}=\mathbf{n}_{S^2}\times\nabla f$ with a real-valued function $f$ on $S^2$ satisfies $\mathrm{div}\,\mathbf{v}=0$ on $S^2$.

\textbf{Step 1:} $\omega_{=1}(t)=\omega_{0,=1}$.
For $|m|=0,1$, we have
\begin{align*}
  (\cos\theta\,\partial_\varphi(I+6\Delta^{-1})\omega(t),Y_1^m)_{L^2(S^2)} = 0
\end{align*}
by \eqref{E:Y_Perturb} and $\lambda_2=6$.
Moreover, since $\lambda_1=2$,
\begin{align*}
  (\Delta\omega(t)+2\omega(t),Y_1^m)_{L^2(S^2)} = (-\lambda_1+2)\omega_1^m(t) = 0.
\end{align*}
Also, for $\psi(t)=\Delta^{-1}\omega(t)$ we have $\mathbf{u}(t) = \mathbf{n}_{S^2}\times\nabla\Delta^{-1}\omega(t) = \mathbf{n}_{S^2}\times\nabla\psi(t)$ and thus
\begin{align*}
  (\mathbf{u}(t)\cdot\nabla\omega(t),Y_1^m)_{L^2(S^2)} = (\mathbf{u}(t)\cdot\nabla\Delta\psi(t),Y_1^m)_{L^2(S^2)} = 0
\end{align*}
by \eqref{E:S2Int_Y1m}.
Hence we take the inner product of \eqref{E:NL_Pert} with $Y_1^m$ to get
\begin{align*}
  \frac{d}{dt}\omega_1^m(t) = \langle\partial_t\omega(t),Y_1^m\rangle_{H_0^1} = 0
\end{align*}
and thus $\omega_1^m(t)=\omega_{0,1}^m$ for $|m|=0,1$, i.e. $\omega_{=1}(t)=\omega_{0,=1}$.

\textbf{Step 2:} derivation of an equation for $\omega_{\geq2}(t)=\omega(t)-\omega_{=1}(t)$.
Let
\begin{align} \label{E:Def_ug2X}
  \begin{aligned}
    X &= \mathbf{n}_{S^2}\times\nabla\Delta^{-1}\omega_{0,=1} = -\frac{1}{2}\mathbf{n}_{S^2}\times\nabla\omega_{0,=1}, \\
    \mathbf{u}_{\geq2}(t) &= \mathbf{n}_{S^2}\times\nabla\Delta^{-1}\omega_{\geq 2}(t).
  \end{aligned}
\end{align}
Here we used $-\Delta\omega_{0,=1}=\lambda_1\omega_{0,=1}=2\omega_{0,=1}$.
Then since $\omega_{=1}(t)=\omega_{0,=1}$,
\begin{align*}
  \omega(t) = \omega_{0,=1}+\omega_{\geq2}(t), \quad \mathbf{u}(t) = X+\mathbf{u}_{\geq 2}(t).
\end{align*}
We substitute these expressions for \eqref{E:NL_Pert}.
Then $\partial_t\omega_{0,=1}=0$ and
\begin{align*}
  \Delta\omega_{0,=1}+2\omega_{0,=1} &= (-\lambda_1+2)\omega_{0,=1} = 0, \\
  \cos\theta\,\partial_\varphi(I+6\Delta^{-1})\omega_{0,=1} &= -\frac{2i}{\sqrt{5}}\,\omega_{0,1}^1Y_2^1+\frac{2i}{\sqrt{5}}\,\omega_{0,1}^{-1}Y_2^{-1}
\end{align*}
by $\lambda_1=2$ and \eqref{E:Y_Perturb}.
Also, since $X\cdot\nabla\omega_{0,=1}=0$ by the expression \eqref{E:Def_ug2X} of $X$,
\begin{align*}
  \mathbf{u}(t)\cdot\nabla\omega(t) &= X\cdot\nabla\omega_{\geq2}(t)+\mathbf{u}_{\geq 2}(t)\cdot\nabla\omega_{0,=1}+\mathbf{u}_{\geq2}(t)\cdot\nabla\omega_{\geq2}(t) \\
  &= X\cdot\nabla(I+2\Delta^{-1})\omega_{\geq 2}(t)+\mathbf{u}_{\geq2}(t)\cdot\nabla\omega_{\geq2}(t),
\end{align*}
where we also used
\begin{align*}
  \mathbf{u}_{\geq 2}(t)\cdot\nabla\omega_{0,=1} &= \{\mathbf{n}_{S^2}\times\nabla\Delta^{-1}\omega_{\geq2}(t)\}\cdot\nabla\omega_{0,=1} \\
  &= -\nabla\Delta^{-1}\omega_{\geq 2}(t)\cdot(\mathbf{n}_{S^2}\times\nabla\omega_{0,=1}) \\
  &= 2\nabla\Delta^{-1}\omega_{\geq 2}(t)\cdot X.
\end{align*}
Hence we obtain
\begin{align} \label{E:Omge2_Eq}
  \begin{aligned}
    \partial_t\omega_{\geq2}(t) &= \nu\{\Delta\omega_{\geq 2}(t)+2\omega_{\geq2}(t)\}-\frac{a}{4}\sqrt{\frac{5}{\pi}}\cos\theta\,\partial_\varphi(I+6\Delta^{-1})\omega_{\geq2}(t) \\
    &\qquad -X\cdot\nabla(I+2\Delta^{-1})\omega_{\geq2}(t)-\mathbf{u}_{\geq2}(t)\cdot\nabla\omega_{\geq 2}(t) \\
    &\qquad +\frac{ia}{2\sqrt{\pi}}\,\omega_{0,1}^1Y_2^1-\frac{ia}{2\sqrt{\pi}}\,\omega_{0,1}^{-1}Y_2^{-1},
  \end{aligned}
\end{align}
where the last two terms are stationary source terms due to the effect of the nondissipative part $\omega_{0,=1}$ of the solution $\omega(t)$ through the perturbation operator.

\textbf{Step 3:} expression of $X$.
By \eqref{E:Om_Conj}, we have
\begin{align*}
  \omega_{0,1}^0\in\mathbb{R}, \quad \omega_{0,1}^1 = \mathrm{Re}[\omega_{0,1}^1]+i\,\mathrm{Im}[\omega_{0,1}^1], \quad \omega_{0,1}^{-1} = -\mathrm{Re}[\omega_{0,1}^1]+i\,\mathrm{Im}[\omega_{0,1}^1].
\end{align*}
Then since
\begin{align*}
  Y_1^0 = \frac{1}{2}\sqrt{\frac{3}{\pi}}\cos\theta, \quad Y_1^{\pm1} = \mp\frac{1}{2}\sqrt{\frac{3}{2\pi}}\sin\theta(\cos\varphi\pm i\sin\varphi),
\end{align*}
we can write $\omega_{0,=1}=\sum_{m=-1}^1\omega_{0,1}^mY_1^m$ as
\begin{align*}
  \omega_{0,=1} = \frac{\omega_{0,1}^0}{2}\sqrt{\frac{3}{\pi}}\cos\theta-\mathrm{Re}[\omega_{0,1}^1]\sqrt{\frac{3}{2\pi}}\sin\theta\cos\varphi+\mathrm{Im}[\omega_{0,1}^1]\sqrt{\frac{3}{2\pi}}\sin\theta\sin\varphi.
\end{align*}
By this equality and \eqref{Pf_SIY:n_gc}, we find that
\begin{align} \label{E:X01_Form}
  X(x) = -\frac{1}{2}\mathbf{n}_{S^2}(x)\times\nabla\omega_{0,=1}(x) = \mathbf{a}\times x, \quad x\in S^2,
\end{align}
where $\mathbf{a}=(a_1,a_2,a_3)\in\mathbb{R}^3$ is given by
\begin{align} \label{E:X01_a}
  a_1 = -\frac{1}{2}\sqrt{\frac{3}{2\pi}}\,\mathrm{Re}[\omega_{0,1}^1], \quad a_2 = \frac{1}{2}\sqrt{\frac{3}{2\pi}}\,\mathrm{Im}[\omega_{0,1}^1], \quad a_3 = \frac{1}{4}\sqrt{\frac{3}{\pi}}\,\omega_{0,1}^0.
\end{align}
In particular, $X$ is a Killing vector field on $S^2$.

\textbf{Step 4:} estimate for $\omega_{\geq3}(t)$.
We take the inner product of \eqref{E:Omge2_Eq} with
\begin{align*}
  (I+6\Delta^{-1})\omega_{\geq2}(t) = \sum_{n\geq 2}\sum_{m=-n}^m\left(1-\frac{6}{\lambda_n}\right)\omega_n^m(t)Y_n^m.
\end{align*}
Then since the right-hand side is in fact the summation for $n\geq 3$ by $\lambda_2=6$,
\begin{align*}
  \bigl(Y_2^m,(I+6\Delta^{-1})\omega_{\geq2}(t)\bigr)_{L^2(S^2)} = 0, \quad m=\pm1.
\end{align*}
Also, it follows from $\lambda_2=6$ and $\lambda_n\geq12$ for $n\geq 3$ that
\begin{align*}
  \langle\partial_t\omega_{\geq2}(t),(I+6\Delta^{-1})\omega_{\geq 2}(t)\rangle_{H_0^1} &\geq \frac{1}{4}\frac{d}{dt}\|\omega_{\geq3}(t)\|_{L^2(S^2)}^2, \\
  \bigl(\Delta\omega_{\geq2}(t)+2\omega_{\geq2}(t),(I+6\Delta^{-1})\omega_{\geq2}(t)\bigr)_{L^2(S^2)} &\leq -5\|\omega_{\geq3}(t)\|_{L^2(S^2)}^2.
\end{align*}
Noting that $(I+6\Delta^{-1})\omega_{\geq2}(t)$ is real-valued and $\cos\theta$ is independent of $\varphi$, we carry out integration by parts with respect to $\varphi$ to find that
\begin{align*}
  \bigl(\cos\theta\,\partial_\varphi(I+6\Delta^{-1})\omega_{\geq2}(t),(I+6\Delta^{-1})\omega_{\geq2}(t)\bigr)_{L^2(S^2)} = 0.
\end{align*}
Since $\mathbf{u}_{\geq2}(t)\cdot\nabla\Delta^{-1}\omega_{\geq2}(t)=0$ by the definition \eqref{E:Def_ug2X} of $\mathbf{u}_{\geq2}(t)$, we have
\begin{align*}
  &\bigl(\mathbf{u}_{\geq 2}(t)\cdot\nabla\omega_{\geq2}(t),(I+6\Delta^{-1})\omega_{\geq2}(t)\bigr)_{L^2(S^2)} \\
  &\qquad = \bigl(\mathbf{u}_{\geq2}\cdot\nabla(I+6\Delta^{-1})\omega_{\geq2}(t),(I+6\Delta^{-1})\omega_{\geq2}(t)\bigr)_{L^2(S^2)} = 0
\end{align*}
by integration by parts and $\mathrm{div}\,\mathbf{u}_{\geq2}(t)=0$.
Also, since $\omega_{\geq2}(t)$ is real-valued and $X$ is of the form \eqref{E:X01_Form}, we can apply \eqref{E:S2Int_Kil} to get
\begin{align*}
  \bigl(X\cdot\nabla\Delta^{-1}\omega_{\geq2}(t),\omega_{\geq2}(t)\bigr)_{L^2(S^2)} = \bigl(X\cdot\nabla\omega_{\geq2}(t),\Delta^{-1}\omega_{\geq2}(t)\bigr)_{L^2(S^2)} = 0.
\end{align*}
We further observe that
\begin{align*}
  \bigl(X\cdot\nabla\omega_{\geq2}(t),\omega_{\geq2}(t)\bigr)_{L^2(S^2)} = \bigl(X\cdot\nabla\Delta^{-1}\omega_{\geq2}(t),\Delta^{-1}\omega_{\geq2}(t)\bigr)_{L^2(S^2)} = 0
\end{align*}
by integration by parts and $\mathrm{div}\,X=0$.
Hence
\begin{align*}
  \bigl(X\cdot\nabla(I+2\Delta^{-1})\omega_{\geq2}(t),(I+6\Delta^{-1})\omega_{\geq2}(t)\bigr)_{L^2(S^2)} = 0.
\end{align*}
Now we apply the above relations to the inner product of \eqref{E:Omge2_Eq} with $(I+6\Delta^{-1})\omega_{\geq2}(t)$.
Then we find that
\begin{align*}
  \frac{1}{4}\frac{d}{dt}\|\omega_{\geq3}(t)\|_{L^2(S^2)}^2 \leq -5\nu\|\omega_{\geq3}(t)\|_{L^2(S^2)}^2
\end{align*}
and thus the estimate \eqref{E:NP_1and3} for $\omega_{\geq3}(t)$ follows (note that $\omega_{\geq3}(0)=\omega_{0,\geq3}$).

\textbf{Step 5:} derivation of ODEs for $\omega_2^m(t)$, $|m|=0,1,2$.
We set
\begin{align} \label{E:Def_u2ge3}
  \begin{aligned}
    \mathbf{u}_{=2}(t) &= \mathbf{n}_{S^2}\times\nabla\Delta^{-1}\omega_{=2}(t) = -\frac{1}{6}\mathbf{n}_{S^2}\times\nabla\omega_{=2}(t), \\
    \mathbf{u}_{\geq3}(t) &= \mathbf{n}_{S^2}\times\nabla\Delta^{-1}\omega_{\geq3}(t).
  \end{aligned}
\end{align}
Here we used $-\Delta\omega_{=2}(t)=\lambda_2\omega_{=2}(t)=6\omega_{=2}(t)$.
We substitute
\begin{align*}
  \omega_{\geq2}(t) = \omega_{=2}(t)+\omega_{\geq3}(t), \quad \mathbf{u}_{\geq2}(t) = \mathbf{u}_{=2}(t)+\mathbf{u}_{\geq3}(t)
\end{align*}
for the right-hand side of \eqref{E:Omge2_Eq}.
Then since $-\Delta\omega_{=2}(t)=6\omega_{=2}(t)$,
\begin{align*}
  \Delta\omega_{\geq2}(t)+2\omega_{\geq2}(t) &= -4\omega_{=2}(t)+\{\Delta\omega_{\geq3}(t)+2\omega_{\geq3}(t)\}, \\
  (I+6\Delta^{-1})\omega_{\geq2}(t) &= (I+6\Delta^{-1})\omega_{\geq3}(t), \\
  X\cdot\nabla(I+2\Delta^{-1})\omega_{\geq2}(t) &= \frac{2}{3}X\cdot\nabla\omega_{=2}(t)+X\cdot\nabla(I+2\Delta^{-1})\omega_{\geq3}(t).
\end{align*}
Also, since $\mathbf{u}_{=2}(t)\cdot\nabla\omega_{=2}(t)=0$ by the expression \eqref{E:Def_u2ge3} of $\mathbf{u}_{=2}(t)$,
\begin{align*}
  \mathbf{u}_{\geq2}(t)\cdot\nabla\omega_{\geq2}(t) &= \mathbf{u}_{=2}(t)\cdot\nabla\omega_{\geq3}(t)+\mathbf{u}_{\geq3}(t)\cdot\nabla\omega_{=2}(t)+\mathbf{u}_{\geq3}(t)\cdot\nabla\omega_{\geq3}(t) \\
  &= \mathbf{u}_{=2}(t)\cdot\nabla(I+6\Delta^{-1})\omega_{\geq3}(t)+\mathbf{u}_{\geq3}(t)\cdot\nabla\omega_{\geq3}(t),
\end{align*}
where we also used
\begin{align*}
  \mathbf{u}_{\geq3}(t)\cdot\nabla\omega_{=2}(t) &= \{\mathbf{n}_{S^2}\times\nabla\Delta^{-1}\omega_{\geq3}(t)\}\cdot\nabla\omega_{=2}(t) \\
  &= -\nabla\Delta^{-1}\omega_{\geq3}(t)\cdot\{\mathbf{n}_{S^2}\times\nabla\omega_{=2}(t)\} \\
  &= 6\nabla\Delta^{-1}\omega_{\geq3}(t)\cdot\mathbf{u}_{=2}(t).
\end{align*}
Hence we have
\begin{align} \label{E:Om_2all_Eq}
  \begin{aligned}
    \partial_t\omega_{\geq2}(t) &= -4\nu\omega_{=2}(t)-\frac{2}{3}X\cdot\nabla\omega_{=2}(t)-\mathbf{u}_{=2}(t)\cdot\nabla(I+6\Delta^{-1})\omega_{\geq3}(t) \\
    &\qquad +\frac{ia}{2\sqrt{\pi}}\,\omega_{0,1}^1Y_2^1-\frac{ia}{2\sqrt{\pi}}\,\omega_{0,1}^{-1}Y_2^{-1}+f_{\geq3}(t),
  \end{aligned}
\end{align}
where
\begin{align} \label{E:Def_fge3}
  \begin{aligned}
    f_{\geq3}(t) &= \nu\{\Delta\omega_{\geq3}(t)+2\omega_{\geq3}(t)\}-\frac{a}{4}\sqrt{\frac{5}{\pi}}\cos\theta\,\partial_\varphi(I+6\Delta^{-1})\omega_{\geq3}(t) \\
    &\qquad -X\cdot\nabla(I+2\Delta^{-1})\omega_{\geq3}(t)-\mathbf{u}_{\geq3}(t)\cdot\nabla\omega_{\geq3}(t).
  \end{aligned}
\end{align}
Noe we take the inner product of \eqref{E:Om_2all_Eq} with $Y_2^m$, $|m|=0,1,2$.
Then
\begin{align} \label{E:O2m_Inner}
  \begin{aligned}
    \frac{d}{dt}\omega_2^m(t) &= -4\nu\omega_2^m(t)-\frac{2}{3}(X\cdot\nabla\omega_{=2}(t),Y_2^m)_{L^2(S^2)} \\
    &\qquad -(\mathbf{u}_{=2}(t)\cdot\nabla(I+6\Delta^{-1})\omega_{\geq3}(t),Y_2^m)_{L^2(S^2)} \\
    &\qquad +\frac{ia}{2\sqrt{\pi}}\,\omega_{0,1}^1\delta_{1,m}-\frac{ia}{2\sqrt{\pi}}\,\omega_{0,1}^{-1}\delta_{-1,m}+(f_{\geq3}(t),Y_2^m)_{L^2(S^2)},
  \end{aligned}
\end{align}
where $\delta_{j,k}$ is the Kronecker delta.
Let us calculate the above inner products.
We observe by integration by parts and $\mathrm{div}\,\mathbf{u}_{=2}(t)=0$ that
\begin{align*}
  (\mathbf{u}_{=2}(t)\cdot\nabla(I+6\Delta^{-1})\omega_{\geq3}(t),Y_2^m)_{L^2(S^2)} &= \int_{S^2}\{\mathbf{u}_{=2}(t)\cdot\nabla(I+6\Delta^{-1})\omega_{\geq3}(t)\}\overline{Y_2^m}\,d\mathcal{H}^2 \\
  &= -\int_{S^2}(I+6\Delta^{-1})\omega_{\geq3}(t)\Bigl(\mathbf{u}_{=2}(t)\cdot\nabla\overline{Y_2^m}\Bigr)\,d\mathcal{H}^2.
\end{align*}
Recall that here we use the complex spherical harmonics and abuse the notations of the inner and vector products in $\mathbb{R}^3$ (see \eqref{E:Abu_Pr}).
We write this expression as
\begin{align} \label{E:U2_Og3_Y2m}
  \begin{gathered}
    (\mathbf{u}_{=2}(t)\cdot\nabla(I+6\Delta^{-1})\omega_{\geq3}(t),Y_2^m)_{L^2(S^2)} = \sum_{k=-2}^2M_{m,k}(t)\omega_2^k(t), \\
    M_{m,k}(t) = \frac{1}{6}\int_{S^2}(I+6\Delta^{-1})\omega_{\geq3}(t)\Bigl\{(\mathbf{n}_{S^2}\times\nabla Y_2^k)\cdot\nabla\overline{Y_2^m}\Bigr\}\,d\mathcal{H}^2
  \end{gathered}
\end{align}
by using \eqref{E:Def_u2ge3} and $\omega_{=2}(t)=\sum_{m=-2}^2\omega_2^m(t)Y_2^m$.
For the inner product of $f_{\geq3}(t)$ given by \eqref{E:Def_fge3} with $Y_2^m$, we see that
\begin{align*}
  (\Delta\omega_{\geq3}(t)+2\omega_{\geq3}(t),Y_2^m)_{L^2(S^2)} &= 0, \\
  (\cos\theta\,\partial_\varphi(I+6\Delta^{-1})\omega_{\geq3}(t),Y_2^m)_{L^2(S^2)} &= \frac{1}{2}ima_3^m\omega_3^m(t)
\end{align*}
by $\omega_{\geq3}(t)=\sum_{n\geq3}\sum_{m=-n}^n\omega_n^m(t)Y_n^m$, \eqref{E:Y_Perturb}, and $\lambda_3=12$.
Also,
\begin{align*}
  (\mathbf{u}_{\geq3}(t)\cdot\nabla\omega_{\geq3}(t),Y_2^m)_{L^2(S^2)} &= \int_{S^2}\{\mathbf{u}_{\geq3}(t)\cdot\nabla\omega_{\geq3}(t)\}\overline{Y_2^m}\,d\mathcal{H}^2 \\
  &= -\int_{S^2}\omega_{\geq3}(t)\Bigl(\mathbf{u}_{\geq3}(t)\cdot\nabla\overline{Y_2^m}\Bigr)\,d\mathcal{H}^2
\end{align*}
by integration by parts and $\mathrm{div}\,\mathbf{u}_{\geq3}(t)=0$.
We also have
\begin{align*}
  (X\cdot\nabla(I+2\Delta^{-1})\omega_{\geq3}(t),Y_2^m)_{L^2(S^2)} &= \int_{S^2}\{X\cdot\nabla(I+2\Delta^{-1})\omega_{\geq3}(t)\}\overline{Y_2^m}\,d\mathcal{H^2} \\
  &= -\int_{S^2}(I+2\Delta^{-1})\omega_{\geq3}(t)\Bigl(X\cdot\nabla\overline{Y_2^m}\Bigr)\,d\mathcal{H}^2 \\
  &= -\bigl((I+2\Delta^{-1})\omega_{\geq3}(t),X\cdot\nabla Y_2^m\bigr)_{L^2(S^2)}
\end{align*}
by integration by parts, $\mathrm{div}\,X=0$, and $X\cdot\nabla\overline{Y_2^m}=\overline{X\cdot\nabla Y_2^m}$ since $X$ is real.
Then since $X$ is of the form \eqref{E:X01_Form}, we see by \eqref{E:Y2m_Kil} and \eqref{E:Y2mK_minus} that
\begin{align*}
  X\cdot\nabla Y_2^m\in\mathrm{span}\{Y_2^0, Y_2^{\pm1}, Y_2^{\pm2}\}, \quad |m| = 0,1,2.
\end{align*}
By this fact, $\omega_{\geq3}(t)=\sum_{n\geq3}\sum_{m=-n}^n\omega_n^m(t)Y_n^m$, and \eqref{E:Def_Laps}, we get
\begin{align*}
  (X\cdot\nabla(I+2\Delta^{-1})\omega_{\geq3}(t),Y_2^m)_{L^2(S^2)} = -\bigl((I+2\Delta^{-1})\omega_{\geq3}(t),X\cdot\nabla Y_2^m\bigr)_{L^2(S^2)} = 0.
\end{align*}
Thus, noting that $f_{\geq3}(t)$ is given by \eqref{E:Def_fge3}, we see by the above equalities that
\begin{align} \label{E:fg3_Y2m}
  (f_{\geq3}(t),Y_2^m)_{L^2(S^2)} = -\frac{a}{8}\sqrt{\frac{5}{\pi}}\,ima_3^m\omega_3^m(t)+\int_{S^2}\omega_{\geq3}(t)\Bigl(\mathbf{u}_{\geq3}(t)\cdot\nabla\overline{Y_2^m}\Bigr)\,d\mathcal{H}^2.
\end{align}
To compute the inner product of $X\cdot\nabla\omega_{=2}(t)$ with $Y_2^m$, we see that we can use \eqref{E:Y2m_Kil} and \eqref{E:Y2mK_minus} since $X$ is of the form \eqref{E:X01_Form}.
Hence
\begin{align*}
  X\cdot\nabla\omega_{=2}(t) = \sum_{m=-2}^m\omega_2^m(t)(X\cdot\nabla Y_2^m) = \sum_{m=-2}^2\eta_2^m(t)Y_2^m,
\end{align*}
where $\eta_2^m(t)=(X\cdot\nabla\omega_{=2}(t),Y_2^m)_{L^2(S^2)}$ is given by
\begin{align*}
  \eta_2^2(t) &= 2ia_3\omega_2^2(t)+(ia_1+a_2)\omega_2^1(t), \\
  \eta_2^1(t) &= (ia_1-a_2)\omega_2^2(t)+ia_3\omega_2^1(t)+\frac{\sqrt{6}}{2}(ia_1+a_2)\omega_2^0(t), \\
  \eta_2^0(t) &= \frac{\sqrt{6}}{2}(ia_1-a_2)\omega_2^1(t)+\frac{\sqrt{6}}{2}(ia_1+a_2)\omega_2^{-1}(t), \\
  \eta_2^{-1}(t) &= \frac{\sqrt{6}}{2}(ia_1-a_2)\omega_2^0(t)-ia_3\omega_2^{-1}(t)+(ia_1+a_2)\omega_2^{-2}(t), \\
  \eta_2^{-2}(t) &= (ia_1-a_2)\omega_2^{-1}(t)-2ia_3\omega_2^{-2}(t).
\end{align*}
Moreover, since
\begin{align*}
  ia_1+a_2 = -\frac{i}{2}\sqrt{\frac{3}{2\pi}}\,\omega_{0,1}^1, \quad ia_1-a_2 = -\frac{i}{2}\sqrt{\frac{3}{2\pi}}\,\overline{\omega_{0,1}^1}
\end{align*}
by \eqref{E:X01_a}, we have
\begin{align*}
  \eta_2^2(t) &= \frac{i}{2}\sqrt{\frac{3}{\pi}}\,\omega_{0,1}^0\,\omega_2^2(t)-\frac{i}{2}\sqrt{\frac{3}{2\pi}}\,\omega_{0,1}^1\,\omega_2^1(t), \\
  \eta_2^1(t) &= -\frac{i}{2}\sqrt{\frac{3}{2\pi}}\,\overline{\omega_{0,1}^1}\,\omega_2^2(t)+\frac{i}{4}\sqrt{\frac{3}{\pi}}\,\omega_{0,1}^0\,\omega_2^1(t)-\frac{3i}{4\sqrt{\pi}}\,\omega_{0,1}^1\,\omega_2^0(t), \\
  \eta_2^0(t) &= -\frac{3i}{4\sqrt{\pi}}\,\overline{\omega_{0,1}^1}\,\omega_2^1(t)-\frac{3i}{4\sqrt{\pi}}\,\omega_{0,1}^1\,\omega_2^{-1}(t), \\
  \eta_2^{-1}(t) &= -\frac{3i}{4\sqrt{\pi}}\,\overline{\omega_{0,1}^1}\,\omega_2^0(t)-\frac{i}{4}\sqrt{\frac{3}{\pi}}\,\omega_{0,1}^0\,\omega_2^{-1}(t)-\frac{i}{2}\sqrt{\frac{3}{2\pi}}\,\omega_{0,1}^1\,\omega_2^{-2}(t), \\
  \eta_2^{-2}(t) &= -\frac{i}{2}\sqrt{\frac{3}{2\pi}}\,\overline{\omega_{0,1}^1}\,\omega_2^{-1}(t)-\frac{i}{2}\sqrt{\frac{3}{\pi}}\,\omega_{0,1}^0\,\omega_2^{-2}(t).
\end{align*}
Hence we can write
\begin{align} \label{E:XgO2_Y2m}
  \frac{2}{3}(X\cdot\nabla\omega_{=2}(t),Y_2^m)_{L^2(S^2)} = \frac{2}{3}\eta_2^m(t) = \sum_{k=-2}^2iA_{m,k}\omega_2^k(t),
\end{align}
where we set $\alpha=\omega_{0,1}^1/2\sqrt{6\pi}\in\mathbb{C}$ and $b=\omega_{0,1}^0/2\sqrt{3\pi}\in\mathbb{R}$ as in \eqref{E:Def_Alb} and define
\begin{align} \label{E:Def_Amk}
  \bm{A} =
  \begin{pmatrix}
    A_{2,2} & \cdots & A_{2,-2} \\
    \vdots & \ddots & \vdots \\
    A_{-2,2} & \cdots & A_{-2,-2}
  \end{pmatrix}
  =
  \begin{pmatrix}
    2b & -2\alpha & 0 & 0 & 0 \\
    -2\bar{\alpha} & b & -\sqrt{6}\,\alpha & 0 & 0 \\
    0 & -\sqrt{6}\,\bar{\alpha} & 0 & -\sqrt{6}\,\alpha & 0 \\
    0 & 0 & -\sqrt{6}\,\bar{\alpha} & -b & -2\alpha \\
    0 & 0 & 0 & -2\bar{\alpha} & -2b
  \end{pmatrix}.
\end{align}
Therefore, we deduce from \eqref{E:O2m_Inner}--\eqref{E:XgO2_Y2m} and $\omega_{=2}(0)=\omega_{0,=2}$ that
\begin{align*}
  \left\{
  \begin{aligned}
    \frac{d}{dt}\omega_2^m(t) &= -4\nu\omega_2^m(t)-\sum_{k=-2}^2\{iA_{m,k}+M_{m,k}(t)\}\omega_2^k(t) \\
    &\qquad +\frac{ia}{2\sqrt{\pi}}\,\omega_{0,1}^1\delta_{1,m}-\frac{ia}{2\sqrt{\pi}}\,\omega_{0,1}^{-1}\delta_{-1,m}+(f_{\geq3}(t),Y_2^m)_{L^2(S^2)}, \quad t>0, \\
    \omega_2^m(0) &= \omega_{0,2}^m
  \end{aligned}
  \right.
\end{align*}
for $|m|=0,1,2$.
We write this system as
\begin{align} \label{E:Om2_ODE}
  \left\{
  \begin{aligned}
    \frac{d}{dt}\bm{\omega}(t) &= -\{4\nu\bm{I}_5+i\bm{A}+\bm{M}(t)\}\bm{\omega}(t)+\bm{f}_{\geq3}(t)+\bm{c}, \quad t>0, \\
    \bm{\omega}(0) &= \bm{\omega}_0,
  \end{aligned}
  \right.
\end{align}
where $\bm{I}_5$ is the $5\times 5$ identity matrix and
\begin{align} \label{E:ODE_Nota}
  \begin{aligned}
    \bm{\omega}(t) &=
    \begin{pmatrix}
      \omega_2^2(t) \\
      \vdots \\
      \omega_2^{-2}(t)
    \end{pmatrix}, \quad
    \bm{\omega}_0 =
    \begin{pmatrix}
      \omega_{0,2}^2 \\
      \vdots \\
      \omega_{0,2}^{-2}
    \end{pmatrix}, \quad
    \bm{M}(t) =
    \begin{pmatrix}
      M_{2,2}(t) & \cdots & M_{2,-2}(t) \\
      \vdots & \ddots & \vdots \\
      M_{-2,2}(t) & \cdots & M_{-2,-2}(t)
    \end{pmatrix}, \\
    \bm{f}_{\geq3}(t) &=
    \begin{pmatrix}
      (f_{\geq3}(t),Y_2^2)_{L^2(S^2)} \\
      \vdots \\
      (f_{\geq3}(t),Y_2^{-2})_{L^2(S^2)}
    \end{pmatrix}, \quad
    \bm{c} = \frac{ia}{2\sqrt{\pi}}
    \begin{pmatrix}
      0 \\
      \omega_{0,1}^1 \\
      0 \\
      -\omega_{0,1}^{-1} \\
      0
    \end{pmatrix}
    = \sqrt{6}\,ia
    \begin{pmatrix}
      0 \\
      \alpha \\
      0 \\
      \bar{\alpha} \\
      0
    \end{pmatrix}.
  \end{aligned}
\end{align}
In the last equality we also used \eqref{E:Om_Conj} and $\alpha=\omega_{0,1}^1/2\sqrt{6\pi}$.

\textbf{Step 5}: asymptotic behavior of $\omega_{=2}(t)$.
Since the matrix $\bm{A}$ given by \eqref{E:Def_Amk} is self-adjoint, all eigenvalues of $\bm{A}$ are real.
Then all eigenvalues of $4\nu\bm{I}_5+i\bm{A}$ have the real part $4\nu\neq0$ and thus they are nonzero.
Hence $4\nu\bm{I}_5+i\bm{A}$ is invertible and we can define
\begin{align} \label{E:Om2V_Lim}
  \bm{\omega}_\infty = (\omega_{2,\infty}^2,\dots,\omega_{2,\infty}^{-2})^T = (4\nu\bm{I}_5+i\bm{A})^{-1}\bm{c}.
\end{align}
Let $\bm{\xi}(t)=\bm{\omega}(t)-\bm{\omega}_\infty$.
Then we see by \eqref{E:Om2_ODE} and \eqref{E:Om2V_Lim} that
\begin{align} \label{E:Xi_ODE}
  \left\{
  \begin{aligned}
    \frac{d}{dt}\bm{\xi}(t) &= -\{4\nu\bm{I}_5+i\bm{A}+\bm{M}(t)\}\bm{\xi}(t)+\bm{f}_{\geq3}(t)-\bm{M}(t)\bm{\omega}_\infty, \quad t>0, \\
    \bm{\xi}(0) &= \bm{\omega}_0-\bm{\omega}_\infty.
  \end{aligned}
  \right.
\end{align}
Let us estimate $\bm{\xi}(t)$.
In what follows, we write $C$ for a general positive constant independent of $t$, $\nu$, $a$, and the initial data $\omega_0\in L_0^2(S^2)$ for \eqref{E:NL_Pert}.
Also, we denote by $\langle\!\langle\cdot,\cdot\rangle\!\rangle$ and $\|\cdot\|$ the inner product and norm of $\mathbb{C}^5$, and use the same notation $\|\cdot\|$ for the Frobenius norm of a matrix.
First we estimate $\bm{M}(t)$ and $\bm{f}_{\geq3}(t)$ given by \eqref{E:U2_Og3_Y2m}, \eqref{E:fg3_Y2m}, and \eqref{E:ODE_Nota}.
Since $\bm{n}_{S^2}$ and $Y_2^m$ are smooth on $S^2$ and $a_3^m$ is a constant given by \eqref{E:Y_Rec}, we have
\begin{align*}
  \|\bm{M}(t)\| &\leq C\|(I+6\Delta^{-1})\omega_{\geq3}(t)\|_{L^2(S^2)}, \\
  \|\bm{f}_{\geq3}(t)\| &\leq C\Bigl(|a||\omega_3^m(t)|+\|\omega_{\geq3}(t)\|_{L^2(S^2)}\|\bm{u}_{\geq3}(t)\|_{L^2(S^2)}\Bigr)
\end{align*}
by H\"{o}lder's inequality.
Moreover, noting that
\begin{align*}
  \omega_{\geq3}(t) = \sum_{n\geq3}\sum_{m=-n}^n\omega_n^m(t)Y_n^2, \quad \mathbf{u}_{\geq3}(t) = \mathbf{n}_{S^2}\times\nabla\Delta^{-1}\omega_{\geq3}(t),
\end{align*}
we use \eqref{E:Def_Laps}, \eqref{E:Lap_Half}, and $\lambda_n\geq12$ for $n\geq3$, and then apply \eqref{E:NP_1and3} to get
\begin{align} \label{E:Mfg3_Est}
  \begin{aligned}
    \|\bm{M}(t)\| &\leq C\|\omega_{\geq3}(t)\|_{L^2(S^2)} \leq Ce^{-10\nu t}\|\omega_{0,\geq3}\|_{L^2(S^2)}, \\
    \|\bm{f}_{\geq3}(t)\| &\leq C\Bigl(|a|\|\omega_{\geq3}(t)\|_{L^2(S^2)}+\|\omega_{\geq3}(t)\|_{L^2(S^2)}^2\Bigr) \\
    &\leq Ce^{-10\nu t}\Bigl(|a|+\|\omega_{0,\geq3}\|_{L^2(S^2)}\Bigr)\|\omega_{0,\geq3}\|_{L^2(S^2)}.
  \end{aligned}
\end{align}
In the last inequality we also used $e^{-10\nu t}\leq1$.
Now we take the real part of the inner product of \eqref{E:Xi_ODE} with $\bm{\xi}(t)$.
Then, noting that $\langle\!\langle\bm{A}\bm{\xi}(t),\bm{\xi}(t)\rangle\!\rangle$ is real since $\bm{A}$ is self-adjoint, we observe by Cauchy--Schwarz's and Young's inequalities that
\begin{align*}
  \frac{1}{2}\frac{d}{dt}\|\bm{\xi}(t)\|^2 &= -4\nu\|\bm{\xi}(t)\|^2-\mathrm{Re}\langle\!\langle\bm{M}(t)\bm{\xi}(t),\bm{\xi}(t)\rangle\!\rangle \\
  &\qquad +\mathrm{Re}\langle\!\langle\bm{f}_{\geq3}(t),\bm{\xi}(t)\rangle\!\rangle-\mathrm{Re}\langle\!\langle\bm{M}(t)\bm{\omega}_\infty,\bm{\xi}(t)\rangle\!\rangle \\
  &\leq -2\nu\|\bm{\xi}(t)\|^2+\|\bm{M}(t)\|\|\bm{\xi}(t)\|^2+\frac{1}{4\nu}\Bigl(\|\bm{f}_{\geq3}(t)\|^2+\|\bm{M}(t)\|^2\|\bm{\omega}_\infty\|^2\Bigr).
\end{align*}
Applying \eqref{E:Mfg3_Est} to this inequality, we get
\begin{align*}
  \frac{d}{dt}\|\bm{\xi}(t)\|^2 \leq \Bigl(-4\nu+Ce^{-10\nu t}\|\omega_{0,\geq3}\|_{L^2(S^2)}\Bigr)\|\bm{\xi}(t)\|^2+\frac{C\sigma}{\nu}e^{-20\nu t},
\end{align*}
where
\begin{align} \label{E:Sigma_V}
  \sigma = \Bigl(|a|^2+\|\omega_{0,\geq3}\|_{L^2(S^2)}^2+\|\bm{\omega}_\infty\|^2\Bigr)\|\omega_{0,\geq3}\|_{L^2(S^2)}^2.
\end{align}
Hence, setting
\begin{align*}
  F(t) = \int_0^t\Bigl(-4\nu+Ce^{-10\nu \tau}\|\omega_{0,\geq3}\|_{L^2(S^2)}\Bigr)\,d\tau = -4\nu t+\frac{C}{10\nu}(1-e^{-10\nu t})\|\omega_{0,\geq3}\|_{L^2(S^2)},
\end{align*}
we observe by the above inequality and $F(t)\geq-4\nu t$ that
\begin{align*}
  \frac{d}{dt}\Bigl(e^{-F(t)}\|\bm{\xi}(t)\|^2\Bigr) \leq \frac{C\sigma}{\nu}e^{-20\nu t-F(t)} \leq \frac{C\sigma}{\nu}e^{-16\nu t}
\end{align*}
and thus
\begin{align*}
  e^{-F(t)}\|\bm{\xi}(t)\|^2 \leq \|\bm{\xi}(0)\|^2+\frac{C\sigma}{16\nu^2}(1-e^{-16\nu t}) \leq \|\bm{\xi}(0)\|^2+\frac{C\sigma}{\nu^2}.
\end{align*}
Therefore, noting that
\begin{align*}
  \bm{\xi}(t) = \bm{\omega}(t)-\bm{\omega}_\infty, \quad \bm{\xi}(0) = \bm{\omega}_0-\bm{\omega}_\infty, \quad F(t) \leq -4\nu t+\frac{C}{\nu}\|\omega_{0,\geq3}\|_{L^2(S^2)},
\end{align*}
we obtain
\begin{align} \label{E:Om2V_Asymp}
  \|\bm{\omega}(t)-\bm{\omega}_\infty\|^2 \leq \left(\|\bm{\omega}_0-\bm{\omega}_\infty\|^2+\frac{C\sigma}{\nu^2}\right)\exp\left(-4\nu t+\frac{C}{\nu}\|\omega_{0,\geq3}\|_{L^2(S^2)}\right).
\end{align}
Now we define $\omega_\infty=\sum_{n=-2}^2\omega_{2,\infty}^mY_2^m\in L_0^2(S^2)$ with $\omega_{2,\infty}^m$ given by \eqref{E:Om2V_Lim}.
Then, noting that $\bm{\omega}(t)$ and $\bm{\omega}_0$ are given by \eqref{E:ODE_Nota} and
\begin{align*}
  \|\bm{w}\| = \|w\|_{L^2(S^2)}, \quad \bm{w} = (w_2^2,\dots,w_2^{-2})^T \in \mathbb{C}^5, \quad w = \sum_{m=-2}^2w_2^mY_2^m,
\end{align*}
we take the square root of \eqref{E:Om2V_Asymp} and use \eqref{E:Sigma_V} to deduce that
\begin{align*}
  \|\omega_{=2}(t)-\omega_\infty\|_{L^2(S^2)} \leq \sigma_1e^{-2\nu t}\left(\|\omega_{0,=2}-\omega_\infty\|_{L^2(S^2)}+\frac{\sigma_2}{\nu}\right)
\end{align*}
for all $t\geq0$, where $\sigma_1$ and $\sigma_2$ are given by \eqref{E:Sigma_12} with constants $C_1,C_2>0$ independent of $t$, $\nu$, $a$, and $\omega_0$.
Hence we conclude that the estimate \eqref{E:NP_2} for $\omega_{=2}(t)$ is valid.

\textbf{Step 6:} expression of $\omega_{2,\infty}^m$, $|m|=0,1,2$.
First note that
\begin{align} \label{E:OmLi_Conj}
  \omega_{2,\infty}^{-m} = \lim_{t\to\infty}\omega_2^{-m}(t) = \lim_{t\to\infty}(-1)^m\overline{\omega_2^m(t)} = (-1)^m\overline{\omega_{2,\infty}^m}, \quad |m|=0,1,2
\end{align}
by \eqref{E:Om_Conj} and \eqref{E:Om2V_Asymp}.
Since $(4\nu\bm{I}_5+i\bm{A})\bm{\omega}_\infty=\bm{c}$ by \eqref{E:Om2V_Lim}, we have
\begin{align}
  (4\nu+2ib)\omega_{2,\infty}^2-2i\alpha\omega_{2,\infty}^1 &= 0, \label{E:OL_Eq_2} \\
  -2i\bar{\alpha}\omega_{2,\infty}^2+(4\nu+ib)\omega_{2,\infty}^1-\sqrt{6}\,i\alpha\omega_{2,\infty}^0 &= \sqrt{6}\,ia\alpha, \label{E:OL_Eq_1} \\
  -\sqrt{6}\,i\bar{\alpha}\omega_{2,\infty}^1+4\nu\omega_{2,\infty}^0-\sqrt{6}\,i\alpha\omega_{2,\infty}^{-1} &= 0 \label{E:OL_Eq_0}
\end{align}
by \eqref{E:Def_Amk} and \eqref{E:ODE_Nota}.
Let $z_k=4\nu+ikb$ for $k=1,2$.
Then
\begin{align} \label{E:OmLi_12}
  \omega_{2,\infty}^2 = \frac{2i\alpha}{z_2}\,\omega_{2,\infty}^1, \quad \omega_{2,\infty}^1 = \frac{\sqrt{6}\,i\alpha z_2}{4|\alpha|^2+z_1z_2}(\omega_{2,\infty}^0+a)
\end{align}
by \eqref{E:OL_Eq_2} and \eqref{E:OL_Eq_1}.
Noting that $\omega_{2,\infty}^{-1}=-\overline{\omega_{2,\infty}^1}$ and $\omega_{2,\infty}^0\in\mathbb{R}$ by \eqref{E:OmLi_Conj} and that $a,b\in\mathbb{R}$, we multiply \eqref{E:OL_Eq_0} by $(4|\alpha|^2+z_1z_2)(4|\alpha|^2+\overline{z_1z_2})$ and use \eqref{E:OmLi_12} to find that
\begin{multline*}
  \{6|\alpha|^2z_2(4|\alpha|^2+\overline{z_1z_2})+4\nu(4|\alpha|^2+z_1z_2)(4|\alpha|^2+\overline{z_1z_2})+6|\alpha|^2\overline{z_2}(4|\alpha|^2+z_1z_2)\}\omega_{2,\infty}^0 \\
  = -6a|\alpha|^2\{z_2(4|\alpha|^2+\overline{z_1z_2})+\overline{z_2}(4|\alpha|^2+z_1z_2)\}.
\end{multline*}
Moreover, by $z_k=4\nu+ikb$ for $k=1,2$ and direct calculations, we have
\begin{align*}
  &6|\alpha|^2z_2(4|\alpha|^2+\overline{z_1z_2})+4\nu(4|\alpha|^2+z_1z_2)(4|\alpha|^2+\overline{z_1z_2})+6|\alpha|^2\overline{z_2}(4|\alpha|^2+z_1z_2) \\
  &\qquad = 16\nu\{16|\alpha|^4+8|\alpha|^2(10\nu^2+b^2)+64\nu^4+20\nu^2b^2+b^4\} \\
  &\qquad = 16\nu(4|\alpha|^2+4\nu^2+b^2)(4|\alpha|^2+16\nu^2+b^2)
\end{align*}
and
\begin{align*}
  z_2(4|\alpha|^2+\overline{z_1z_2})+\overline{z_2}(4|\alpha|^2+z_1z_2) = 32\nu(|\alpha|^2+4\nu^2+b^2).
\end{align*}
Hence
\begin{align} \label{E:OmLi_0}
  \omega_{2,\infty}^0 = -12a\cdot\frac{|\alpha|^2(4\nu^2+|\alpha|^2+b^2)}{(4\nu^2+4|\alpha|^2+b^2)(16\nu^2+4|\alpha|^2+b^2)}
\end{align}
and we conclude by \eqref{E:OmLi_Conj}, \eqref{E:OmLi_12}, \eqref{E:OmLi_0}, and $z_k=4\nu+ikb$ for $k=1,2$ that $\omega_\infty=\sum_{m=-2}^2\omega_{2,\infty}^mY_2^m$ is given by \eqref{E:Def_Alb}--\eqref{E:Def_OmLim}.
The proof of Theorem \ref{T:NLS_Main} is complete.

\section{Proof of Lemma \ref{L:Y2m_Kil}} \label{S:PoYK}
Let us give the proof of \eqref{E:Y2m_Kil} in Lemma \ref{L:Y2m_Kil}.
First we recall that
\begin{align} \label{E:Y2m_Expl}
  \begin{alignedat}{2}
    Y_2^0(\theta) &= C_0(3\cos^2\theta-1), &\quad C_0 &= \frac{1}{4}\sqrt{\frac{5}{\pi}}, \\
    Y_2^{\pm1}(\theta,\varphi) &= \mp C_1\sin\theta\cos\theta\,e^{\pm i\varphi}, &\quad C_1 &= \frac{1}{2}\sqrt{\frac{15}{2\pi}}, \\
    Y_2^{\pm2}(\theta,\varphi) &= C_2\sin^2\theta\,e^{\pm2i\varphi}, &\quad C_2 &= \frac{1}{4}\sqrt{\frac{15}{2\pi}}.
  \end{alignedat}
\end{align}
We use the spherical coordinates $x(\theta,\varphi)$ of $S^2$ given by \eqref{E:Sphrical} so that
\begin{align} \label{E:Tan_Basis}
  \partial_\theta x(\theta,\varphi) =
  \begin{pmatrix}
    \cos\theta\cos\varphi \\ \cos\theta\sin\varphi \\ -\sin\theta
  \end{pmatrix}, \quad
  \partial_\varphi x(\theta,\varphi) =
  \begin{pmatrix}
    -\sin\theta\sin\varphi \\ \sin\theta\cos\varphi \\ 0
  \end{pmatrix}.
\end{align}
In what follows, we suppress the arguments $\theta$ and $\varphi$.
Then
\begin{align*}
  \mathbf{e}_1\times x =
  \begin{pmatrix}
    0 \\ -\cos\theta \\ \sin\theta\sin\varphi
  \end{pmatrix}, \quad
  \mathbf{e}_2\times x =
  \begin{pmatrix}
    \cos\theta \\ 0 \\ -\sin\theta\cos\varphi
  \end{pmatrix}, \quad
  \mathbf{e}_3\times x =
  \begin{pmatrix}
    -\sin\theta\sin\varphi \\ \sin\theta\cos\varphi \\ 0
  \end{pmatrix}
\end{align*}
for the standard basis $\{\mathbf{e}_1,\mathbf{e}_2,\mathbf{e}_3\}$ of $\mathbb{R}^3$.
Hence
\begin{align} \label{E:Inn_KB}
  \begin{alignedat}{2}
    (\mathbf{e}_1\times x)\cdot\partial_\theta x &= -\sin\varphi, &\quad (\mathbf{e}_1\times x)\cdot\partial_\varphi x &= -\sin\theta\cos\theta\cos\varphi, \\
    (\mathbf{e}_2\times x)\cdot\partial_\theta x &= \cos\varphi, &\quad (\mathbf{e}_2\times x)\cdot\partial_\varphi x &= -\sin\theta\cos\theta\sin\varphi, \\
    (\mathbf{e}_3\times x)\cdot\partial_\theta x &= 0, &\quad (\mathbf{e}_3\times x)\cdot\partial_\varphi x &= \sin^2\theta.
  \end{alignedat}
\end{align}
Let us compute $(\mathbf{e}_k\times x)\cdot\nabla Y_2^0$ wth $k=1,2,3$.
By \eqref{E:Grad}, we have
\begin{align*}
  \nabla Y_2^0 = -6C_0\sin\theta\cos\theta\,\partial_\theta x.
\end{align*}
Hence it follows from \eqref{E:Y2m_Expl} and \eqref{E:Inn_KB} that
\begin{align*}
  (\mathbf{e}_1\times x)\cdot\nabla Y_2^0 &= 6C_0\sin\theta\cos\theta\sin\varphi = 6C_0\sin\theta\cos\theta\cdot\frac{e^{i\varphi}-e^{-i\varphi}}{2i} \\
  &= \frac{3C_0}{iC_1}(-Y_2^1-Y_2^{-1}) = \frac{\sqrt{6}\,i}{2}(Y_2^1+Y_2^{-1}).
\end{align*}
Similarly, using \eqref{E:Y2m_Expl}, \eqref{E:Inn_KB}, and $\cos\varphi=(e^{i\varphi}+e^{-i\varphi})/2$, we obtain
\begin{align*}
  (\mathbf{e}_2\times x)\cdot\nabla Y_2^0 = \frac{\sqrt{6}}{2}(Y_2^1-Y_2^{-1}), \quad (\mathbf{e}_3\times x)\cdot\nabla Y_2^0 = 0.
\end{align*}
Next we consider $(\mathbf{e}_k\times x)\cdot\nabla Y_2^1$, $k=1,2,3$.
Since
\begin{align*}
  \nabla Y_2^1 = -C_1e^{i\varphi}\left\{(\cos^2\theta-\sin^2\theta)\partial_\theta x+\frac{i\cos\theta}{\sin\theta}\partial_\varphi x\right\}
\end{align*}
by \eqref{E:Grad}, we see by \eqref{E:Inn_KB} that (see also our notation \eqref{E:Abu_Pr})
\begin{align*}
  (\mathbf{e}_1\times x)\cdot\nabla Y_2^1 = C_1e^{i\varphi}\{(\cos^2\theta-\sin^2\theta)\sin\varphi+i\cos^2\theta\cos\varphi\}.
\end{align*}
Moreover, we use
\begin{align*}
  e^{i\varphi}\sin\varphi = -\frac{i}{2}(e^{2i\varphi}-1), \quad e^{i\varphi}\cos\varphi=\frac{1}{2}(e^{2i\varphi}+1), \quad \sin^2\theta = 1-\cos^2\theta
\end{align*}
and then apply \eqref{E:Y2m_Expl} to find that
\begin{align*}
  (\mathbf{e}_1\times x)\cdot\nabla Y_2^1 &= \frac{iC_1}{2}(\sin^2\theta\,e^{2i\varphi}+3\cos^2\theta-1) \\
  &= \frac{i}{2}\left(\frac{C_1}{C_2}Y_2^2+\frac{C_1}{C_0}Y_2^0\right) = i\left(Y_2^2+\frac{\sqrt{6}}{2}Y_2^0\right).
\end{align*}
In the same way, we obtain
\begin{align*}
  (\mathbf{e}_2\times x)\cdot\nabla Y_2^1 = Y_2^2-\frac{\sqrt{6}}{2}Y_2^0, \quad (\mathbf{e}_3\times x)\cdot\nabla Y_2^1 = iY_2^1.
\end{align*}
Now let us calculate $(\mathbf{e}_k\times x)\cdot\nabla Y_2^2$, $k=1,2,3$.
By \eqref{E:Grad}, we have
\begin{align*}
  \nabla Y_2^2 = 2C_2e^{2i\varphi}(\sin\theta\cos\theta\,\partial_\theta x+i\partial_\varphi x).
\end{align*}
Hence we observe by \eqref{E:Inn_KB} that
\begin{align*}
  (\mathbf{e}_1\times x)\cdot\nabla Y_2^2 = -2C_2e^{2i\varphi}\sin\theta\cos\theta(\sin\varphi+i\cos\varphi).
\end{align*}
Moreover, since $\sin\varphi+i\cos\varphi=i(\cos\theta-i\sin\theta)=ie^{-i\varphi}$, we get
\begin{align*}
  (\mathbf{e}_1\times x)\cdot\nabla Y_2^2 = -2iC_2\sin\theta\cos\theta\,e^{i\varphi} = 2i\frac{C_2}{C_1}Y_2^1 = iY_2^1
\end{align*}
by \eqref{E:Y2m_Expl}.
Calculating similarly, we find that
\begin{align*}
  (\mathbf{e}_2\times x)\cdot\nabla Y_2^2 = -Y_2^1, \quad (\mathbf{e}_3\times x)\cdot\nabla Y_2^2 = 2iY_2^2.
\end{align*}
Applying the above results to $\mathbf{a}=(a_1,a_2,a_3)^T=\sum_{k=1}^3a_k\mathbf{e}_k$, we obtain \eqref{E:Y2m_Kil}.

\section{Appendix A: One-jet case} \label{S:Onejet}
We briefly explain the behavior of a perturbation for the one-jet Kolmogorov type flow
\begin{align*}
  \omega_1^a(\theta,\varphi) = aY_1^0(\theta) = \frac{a}{2}\sqrt{\frac{3}{\pi}}\cos\theta.
\end{align*}
The nonlinear stability problem for the vorticity equation \eqref{E:Vort} around the one-jet Kolmogorov type flow is (see \cite{Miu21pre} for the derivation of the perturbation operator)
\begin{align} \label{E:NL_Onejet}
  \left\{
  \begin{aligned}
    \partial_t\omega &= \nu(\Delta\omega+2\omega)-\frac{a}{4}\sqrt{\frac{3}{\pi}}\,\partial_\varphi(I+2\Delta^{-1})\omega-\mathbf{u}\cdot\nabla\omega \quad\text{on}\quad S^2\times(0,\infty), \\
    \mathbf{u} &= \mathbf{n}_{S^2}\times\nabla\Delta^{-1}\omega \quad\text{on}\quad S^2\times(0,\infty), \\
    \omega|_{t=0} &= \omega_0 \quad\text{on}\quad S^2.
  \end{aligned}
  \right.
\end{align}
For a real-valued $\omega_0\in L_0^2(S^2)$, we can prove the global existence and uniqueness of a weak solution $\omega$ to \eqref{E:NL_Onejet} in the class \eqref{E:WS_Class} by the Galerkin and energy methods.
Also, taking the inner product of \eqref{E:NL_Onejet} with $Y_1^m$, $|m|=0,1$ and using $-\Delta Y_1^m=2Y_1^m$ and \eqref{E:S2Int_Y1m}, we can show $\omega_{=1}(t)=\omega_{0,=1}$ for all $t\geq0$ as in Step 1 of the proof of Theorem \ref{T:NLS_Main}.
Moreover, we take the inner product of \eqref{E:NL_Onejet} with $(I+2\Delta^{-1})\omega(t)$.
Then
\begin{align*}
  \langle\partial_t\omega(t),(I+2\Delta^{-1})\omega(t)\rangle_{H_0^1} &\geq \frac{1}{3}\frac{d}{dt}\|\omega_{\geq2}(t)\|_{L^2(S^2)}^2, \\
  \bigl(\Delta\omega(t)+2\omega(t),(I+2\Delta^{-1})\omega(t)\bigr)_{L^2(S^2)} &\leq -\frac{8}{3}\|\omega_{\geq2}(t)\|_{L^2(S^2)}^2
\end{align*}
by $\lambda_1=2$ and $\lambda_n\geq6$ for $n\geq2$.
Also, noting that $(I+2\Delta^{-1})\omega(t)$ is real-valued, we have
\begin{align*}
  \bigl(\partial_\varphi(I+2\Delta^{-1})\omega(t),(I+2\Delta^{-1})\omega(t)\bigr)_{L^2(S^2)} = 0
\end{align*}
by integration by parts and, since $\mathbf{u}(t)\cdot\nabla\Delta^{-1}\omega(t)=0$ and $\mathrm{div}\,\mathbf{u}(t)=0$,
\begin{align*}
  \bigl(\mathbf{u}(t)\cdot\nabla\omega(t),(I+2\Delta^{-1})\omega(t)\bigr)_{L^2(S^2)} = \bigl(\mathbf{u}\cdot\nabla(I+2\Delta^{-1})\omega(t),(I+2\Delta^{-1})\omega(t)\bigr)_{L^2(S^2)} = 0.
\end{align*}
Therefore, we obtain
\begin{align*}
  \frac{1}{3}\frac{d}{dt}\|\omega_{\geq2}(t)\|_{L^2(S^2)}^2 \leq -\frac{8}{3}\|\omega_{\geq2}(t)\|_{L^2(S^2)}^2,
\end{align*}
which implies that
\begin{align*}
  \|\omega_{\geq2}(t)\|_{L^2(S^2)}^2 \leq e^{-8\nu t}\|\omega_{\geq2}(0)\|_{L^2(S^2)}^2 = e^{-8\nu t}\|\omega_{0,\geq2}\|_{L^2(S^2)}^2,
\end{align*}
i.e. $\|\omega_{\geq2}(t)\|_{L^2(S^2)}\leq e^{-4\nu t}\|\omega_{0,\geq2}\|_{L^2(S^2)}$ for all $t\geq0$.

\section{Appendix B: Transformation into an equation with Coriolis force} \label{S:Rotate}
In this section we see that solutions $\omega$ to \eqref{E:NL_Pert} and $\zeta$ to \eqref{E:NLP_Rot} are related by \eqref{E:Sol_CoV}.

We use the spherical coordinates $x(\theta,\varphi)$ of $S^2$ given by \eqref{E:Sphrical} and write $u(\theta,\varphi)$ instead of $(u\circ x)(\theta,\varphi)$ for a function $u$ on $S^2$.
Hence $\partial_\theta x$ and $\partial_\varphi x$ are of the form \eqref{E:Tan_Basis} and
\begin{align} \label{E:GDN_Sph}
  \begin{aligned}
    \nabla u(\theta,\varphi) &= \partial_\theta u(\theta,\varphi)\partial_\theta x(\theta,\varphi)+\frac{\partial_\varphi u(\theta,\varphi)}{\sin^2\theta}\,\partial_\varphi x(\theta,\varphi), \\
    \Delta u(\theta,\varphi) &= \frac{1}{\sin\theta}\partial_\theta\bigl(\sin\theta\,\partial_\theta u(\theta,\varphi)\bigr)+\frac{1}{\sin^2\theta}\,\partial_\varphi^2u(\theta,\varphi), \\
    \mathbf{n}_{S^2}(\theta,\varphi) &= (\sin\theta\cos\varphi, \sin\theta\sin\varphi, \cos\theta)^T.
  \end{aligned}
\end{align}
Also, we easily find that
\begin{align} \label{E:N_times_Th}
  \mathbf{n}_{S^2}(\theta,\varphi)\times\partial_\theta x(\theta,\varphi) = \frac{1}{\sin\theta}\,\partial_\varphi x(\theta,\varphi).
\end{align}
For $\Phi\in\mathbb{R}$ let $R(\Phi)$ be the rotation matrix around the $x_3$-axis by the angle $\Phi$, i.e.
\begin{align*}
  R(\Phi) =
  \begin{pmatrix}
    \cos\Phi & -\sin\Phi & 0 \\
    \sin\Phi & \cos\Phi & 0 \\
    0 & 0 & 1
  \end{pmatrix}.
\end{align*}
Then, for $\mathbf{f}=\partial_\theta x,\partial_\varphi x,\mathbf{n}_{S^2}$ and $\mathbf{a},\mathbf{b}\in\mathbb{R}^3$, we have
\begin{align} \label{E:Vec_Rot}
  \begin{aligned}
    R(\Phi)\mathbf{f}(\theta,\varphi) &= \mathbf{f}(\theta,\varphi+\Phi), \\
    \bigl(R(\Phi)\mathbf{a}\bigr)\cdot\bigl(R(\Phi)\mathbf{b}\bigr) &= \mathbf{a}\cdot\mathbf{b}, \\
    \bigl(R(\Phi)\mathbf{a}\bigr)\times\bigl(R(\Phi)\mathbf{b}\bigr) &= R(\Phi)(\mathbf{a}\times\mathbf{b}).
  \end{aligned}
\end{align}
Let $u$ and $v$ be functions in $L_0^2(S^2)$ such that $u(\theta,\varphi)=v(\theta,\varphi-\Phi)$.
Then
\begin{align} \label{E:Gr_CoV}
  \begin{aligned}
    \nabla u(\theta,\varphi) &= \partial_\theta u(\theta,\varphi)\partial_\theta x(\theta,\varphi)+\frac{\partial_\varphi u(\theta,\varphi)}{\sin^2\theta}\,\partial_\varphi x(\theta,\varphi) \\
    &= \partial_\theta v(\theta,\varphi-\Phi)R(\Phi)\partial_\theta x(\theta,\varphi-\Phi)+\frac{\partial_\varphi v(\theta,\varphi-\Phi)}{\sin^2\theta}\,R(\Phi)\partial_\varphi x(\theta,\varphi-\Phi) \\
    &= R(\Phi)\nabla v(\theta,\varphi-\Phi)
  \end{aligned}
\end{align}
by \eqref{E:GDN_Sph} and \eqref{E:Vec_Rot}.
Moreover, we see by \eqref{E:GDN_Sph} that
\begin{align} \label{E:Lap_CoV}
  \Delta u(\theta,\varphi)=\Delta v(\theta,\varphi-\Phi), \quad \Delta^{-1}u(\theta,\varphi)=\Delta^{-1}v(\theta,\varphi-\Phi).
\end{align}
Now let $\omega$ and $\zeta$ be functions on $S^2\times(0,\infty)$ related by \eqref{E:Sol_CoV}, i.e.
\begin{align*}
  \omega(\theta,\varphi,t) = \zeta(\theta,\varphi-\Omega t,t)+2\Omega\cos\theta.
\end{align*}
Then we observe by \eqref{E:Lap_CoV} and $\Delta\cos\theta=-2\cos\theta$ that
\begin{align} \label{E:CoV_01}
  \begin{aligned}
    \partial_t\omega(\theta,\varphi,t) &= \partial_t\zeta(\theta,\varphi-\Omega t,t)-\Omega\,\partial_\varphi\zeta(\theta,\varphi-\Omega t,t), \\
    (\Delta\omega+2\omega)(\theta,\varphi,t) &= (\Delta\zeta+2\zeta)(\theta,\varphi-\Omega t,t), \\
    \partial_\varphi(I+6\Delta^{-1})\omega(\theta,\varphi,t) &= \partial_\varphi(I+6\Delta^{-1})\zeta(\theta,\varphi-\Omega t,t).
  \end{aligned}
\end{align}
Moreover, we deduce from \eqref{E:Vec_Rot}--\eqref{E:Lap_CoV} and $\Delta^{-1}\cos\theta=-2^{-1}\cos\theta$ that
\begin{align*}
  \nabla\omega(\theta,\varphi,t) &= R(\Omega t)\nabla\zeta(\theta,\varphi-\Omega t,t)-2\Omega\sin\theta\,\partial_\theta x(\theta,\varphi) \\
  &= R(\Omega t)\{\nabla\zeta(\theta,\varphi-\Omega t,t)-2\Omega\sin\theta\,\partial_\theta x(\theta,\varphi-\Omega t)\}, \\
  \nabla\Delta^{-1}\omega(\theta,\varphi,t) &= R(\Omega t)\{\nabla\Delta^{-1}\zeta(\theta,\varphi-\Omega t,t)+\Omega\sin\theta\,\partial_\theta x(\theta,\varphi-\Omega t)\}.
\end{align*}
Thus, setting $\mathbf{u}=\mathbf{n}_{S^2}\times\nabla\Delta^{-1}\omega$ and $\mathbf{v}=\mathbf{n}_{S^2}\times\nabla\Delta^{-1}\zeta$, we have
\begin{align*}
  \mathbf{u}(\theta,\varphi,t) &= \mathbf{n}_{S^2}(\theta,\varphi)\times\nabla\Delta^{-1}\omega(\theta,\varphi,t) \\
  &= \mathbf{n}_{S^2}(\theta,\varphi)\times\bigl(R(\Omega t)\{\nabla\Delta^{-1}\zeta(\theta,\varphi-\Omega t,t)+\Omega\sin\theta\,\partial_\theta x(\theta,\varphi-\Omega t)\}\bigr) \\
  &= R(\Omega t)\{\mathbf{v}(\theta,\varphi-\Omega t,t)+\Omega\,\partial_\varphi x(\theta,\varphi-\Omega t)\}
\end{align*}
by using \eqref{E:Vec_Rot} and then applying \eqref{E:N_times_Th}.
Hence
\begin{align*}
  (\mathbf{u}\cdot\nabla\omega)(\theta,\varphi,t) &= [(\mathbf{v}+\Omega\,\partial_\varphi x)\cdot(\nabla\zeta-2\Omega\sin\theta\,\partial_\theta x)](\theta,\varphi-\Omega t,t)
\end{align*}
by \eqref{E:Vec_Rot}.
Moreover, since $\partial_\varphi x\cdot\partial_\theta x=0$, $\partial_\varphi x\cdot\nabla\zeta=\partial_\varphi\zeta$, and
\begin{align*}
  \mathbf{v}\cdot\partial_\theta x &= (\mathbf{n}_{S^2}\times\nabla\Delta^{-1}\zeta)\cdot\partial_\theta x = -\nabla\Delta^{-1}\zeta\cdot(\mathbf{n}_{S^2}\times\partial_\theta x) \\
  &= -\frac{1}{\sin\theta}\,\nabla\Delta^{-1}\zeta\cdot\partial_\varphi x = -\frac{1}{\sin\theta}\,\partial_\varphi\Delta^{-1}\zeta
\end{align*}
by $(\mathbf{a}\times\mathbf{b})\cdot\mathbf{c}=-\mathbf{b}\cdot(\mathbf{a}\times\mathbf{c})$ for $\mathbf{a},\mathbf{b},\mathbf{c}\in\mathbb{R}^3$ and \eqref{E:N_times_Th}, we get
\begin{align} \label{E:Con_CoV}
  (\mathbf{u}\cdot\nabla\omega)(\theta,\varphi,t) = (\mathbf{v}\cdot\nabla\zeta+\Omega\,\partial_\varphi\zeta+2\Omega\,\partial_\varphi\Delta^{-1}\zeta)(\theta,\varphi-\Omega t,t).
\end{align}
Therefore, for the functions $\omega$ and $\zeta$ related by \eqref{E:Sol_CoV}, we conclude by \eqref{E:CoV_01} and \eqref{E:Con_CoV} that $\omega$ is a solution to \eqref{E:NL_Pert} if and only if $\zeta$ is a solution to \eqref{E:NLP_Rot}.

\section*{Acknowledgments}
The work of the author was supported by Grant-in-Aid for JSPS Fellows No. 19J00693.

\bibliographystyle{abbrv}
\bibliography{NSK_NLS_Ref}

\end{document}